\documentclass{article}[12pt]
\usepackage{amsmath}
\usepackage{amssymb}
\usepackage{amsthm}
\usepackage{color}
\usepackage{sectsty}

\paragraphfont{\normalfont\itshape}
\theoremstyle{remark}
\newtheorem*{remark}{Remark}

\newtheorem{thm}{Theorem}[section]
\newtheorem{lem}{Lemma}[section]
\newtheorem{defn}{Definition}[section]
\newtheorem{cor}{Corollary}[section]
\newtheorem{ex}{Example}[section]

\title{Comments on J.~F.~Ritt's book\\
 ``Integration in Finite Terms''}

\author{Askold Khovanskii\thanks{The work was partially supported by the Canadian Grant No. 156833-17. }}

\begin{document}
\maketitle

\begin{abstract}

The First and Second Liouville's Theorems provide correspondingly criterium  for integrability of elementary functions ``in finite terms'' and criterium for solvability of second order linear differential equations by quadratures. The brilliant book of J.F.~Ritt  contains proofs of these theorems and many other interesting results. This paper was written as comments on the book but one can read it  independently. The first part of the paper contains  modern   proofs of  The First Theorem and  of a generalization of the Second Theorem for linear differential equations of any order. In the second part of the paper we present an outline of topological Galois theory which provides an alternative approach to the problem of solvability of equations in finite terms. The first  section of this part  deals with a topological approach to representability of algebraic functions by radicals  and to the 13-th Hilbert problem. This section is written with  all proofs. Next sections  contain only statements of results and comments on them (basically no proofs are presented there).
\end{abstract}

\tableofcontents
\section{Preface}

I saw  J.F.Ritt's book \cite{[11]} for the first time   in 1969 when I was an undergraduate student.  I just started to work on  topological obstructions to representability of algebraic functions by radicals and on an algebraic version of the 13-th Hilbert problem on representability of algebraic functions of several complex variables by composition of algebraic functions of fewer number of variables. My beloved supervisor Vladimir Igorevich Arnold was  very interested in these questions.

J.F.Ritt's approach,  which uses the theory of complex analytic functions and geometry, was very different from a formal algebraic approach. I was very intrigued and I started to read the book trying to get a feeling about the subject and avoiding all the details for the first reading.

My first impression was that the book was brilliant and  the presented theory was ingenious. Simultaneously  with the reading   I  obtained the very first  results of topological Galois theory. Since then   I have spent a few years  developing it. I had hoped to return back to the book  later, but I never made it (life is life~!).

Even a brief reading  turned out to be very useful. It helped me to formalize the definition of the Liouvillian classes of functions and the definition of the functional differential fields and their extensions. Later this experience  helped me to find an appropriate  definition of the class of Pfaffian   functions playing  the crucial role in a transcendental generalization of real algebraic geometry developed in the book \cite{[5]}.

That is why I was really happy when Michael Singer invited me to write comments for a { reprint} of the  book. I started to read it  again after almost a half century break.
I have to confess  that it was  hard for me to follow  all the details. The reason is that J.F.Ritt uses an old mathematical language (the book was written about  seventy years ago).  Nevertheless I still think
that that the book is brilliant and Liouville's and Ritt's ideas  are ingenious.

In section 2 we present  modern definitions of Liouvillian classes of functions and modern proofs of the First Liouville's Theorem and of the Second Liouville's Theorem. I hope that this modern presentation will help readers understand better the subject and  J.F.Ritt's book.

In the section 3 we present an outline of topological Galois theory which provides an alternative approach to the problem of solvability of equations in finite terms. We use the definition of classes of functions by the list of basic functions and the list of admissible operations presented in the section~\ref{sec:genelem} .

A few words about our proof of the First Liouville's Theorem. All main ideas of the proof are presented in the book. I  tried  to clarify  what is hidden behind the integration used by
Liouville. I think  that there are two statements which were not mentioned explicitly in the book:   1) a closed 1-form with elementary integral whose  possible form was   found by Liouville  is locally invariant under the
Galois group action, assuming that the Galois group is connected; 2) A class of closed 1-forms locally invariant under a connected Lie group action can be described explicitly. In fact all arguments needed for proving the first statement are presented in the book.
Liouville used an explicit integration for description of closed 1-forms locally invariant under a natural action of the additive and the multiplicative groups of complex numbers.

A few words about our proof of the Second Liouville's Theorem. In fact we prove its generalization  applicable to linear homogeneous differential equation of any order.
J.F.Ritt's book  contains basically all results needed for our version of the proof. Only
some statements are missing there (but all arguments needed for their proofs  are presented in some
form in the book).\\

{ \bf Acknowledgement.}
I would like to thank Michael Singer who invited me to write comments for a new edition of the classical J.F.~Ritt's book  for his constant support.  I also am grateful to Fedor Kogan who edited my English and to my wife Tanya Belokrinitskaya who helped me to connect a few papers into one text

\section{Solvability of Equations in Finite Terms.}
\subsection{Introduction}
\medskip

Let $K$ be a subfield of the field of meromorphic functions on a connected domain $U$ of the complex line closed under the differentiation (i.e if $f\in K$ then $f'\in K$). Such field $K$ with the operation of differentiation $f\rightarrow f'$ provides an  example of  {\it functional differential field}.

Liouville's First Theorem suggests  conditions on a function $f$ from a function differential field $K$ which are necessary and sufficient for representability of an indefinite integral of $f$  in {\it generalized elementary functions over $K$}.

Liouville's Second Theorem suggests  conditions on second order homogeneous linear differential equation whose coefficients belong to a  function differential field $K$ which are necessary and sufficient for its solvability  by {\it generalized quadratures  over $K$}.

The Liouville's theory can be generalized to {\it an abstract differential field $K$}, whose elements are not necessarily meromorphic functions (see \cite{[6]}, \cite{[10]}). Abstract algebraic results are not  directly applicable to integrals of elementary functions and to solutions of linear differential equations  which could be multivalued, could  have singularities and so on. For their applications some extra arguments are needed. Such arguments are presented in Chapter I of J.F.Ritt's book.
\smallskip

This section has a following content

In  section~\ref{sec:genelem}  we define functional differential fields, generalized elementary functions and generalized quadratures over such fields. Material of section 1 was inspired by   Chapter I of  the book \cite{[11]}.

In  section~\ref{sec:first} we  prove  Liouville's First Theorem  making use of algebraic  groups actions.
Our proof can be considered as a modernization of the Liouville's proof presented in Chapter II and Chapter III of the  book \cite{[11]}.

In  section~\ref{sec:second} we   prove Liouville's Second Theorem and  its generalizations for homogeneous linear differential equations of any order.  Our proof  uses  slightly   modernized  Liouville's-Ritt's  arguments presented in   Chapters V and VI of J.F. Ritt's book.

\subsection{Generalized Elementary Functions and Generalized Quadratures.}\label{sec:genelem}

\subsubsection{Introduction}

The results presented in this section are inspired by the the material of  Chapter 1 from  Ritt's book \cite{[11]}. We discuss here definitions and general statements related  to functional and abstract differential fields and classes of their extensions including generalized elementary extensions and extensions by generalized quadratures. We follow mainly the  presentation  from the book \cite{[6]}.

 A natural definitions of generalized elementary function and of a function representable by generalized quadratures  over $K$ (see definitions~\ref{def9}, \ref{def11}, \ref{def13}, and \ref{def15} below) are hard to deal with.  In particular they  make use of a  non algebraic operation  of composition of  functions.  Algebraic definitions (see definitions \ref{def3} and \ref{def4} below)   use solution of the simplest differential equations  instead of composition of functions. We explain how the natural definitions can be reduced to the algebraic ones.

\subsubsection{ Differential fields and their extensions}

Let us start with some  pure algebraic definitions.

 \paragraph{Abstract differential fields}

A  field $F$  is said to be
{\it a differential field} if an additive map
$a\rightarrow a'$ is fixed that satisfies the Leibnitz rule $(ab)'=a'b+ab'$.
The element $a'$ is
called the  {\it derivative} of $a$. An element $y\in F$ is called {\it a constant} if $y'=0$.
All constants in $F$ form {\it the field of constants}. We add to the definition of  differential field an extra condition that {\it the field of constants is the field of complex numbers}(for our purpose it is enough to consider fields satisfying this condition).
An element $y\in F$  is said to be: {\it an exponential} of  $a$ if $y'=a'y$; {\it an exponential of integral} of  $a$ if $y'=ay$; {\it a logarithm} of  $a$ if $y'=a'/a$; {\it an integral} of $a$ if $y'=a$. In each of these cases, $y$ is defined only up to an additive or a multiplicative complex constant.

Let $K\subset F$ be a differential subfield in $F$. An element $y$ is said to be an {\it integral over $K$} if $y'=a\in K$. An {\it exponential of integral over $K$},  {\it a logarithm over $K$}, and {\it an integral over $K$} are defined similarly.

Suppose that a differential field $K$ and a set $M$ lie in some differential field $F$.
{\it The adjunction} of the set $M$ to the differential field $K$ is
the minimal differential field $K\langle M\rangle$ containing both the field $K$
and the set $M$. We will refer to the transition from $K$ to $K\langle M\rangle$ as {\it adjoining} the set $M$ to the field $K$.

\paragraph{Generalized elementary extensions}

Let $F\supset K$ be an  extension of a differential field $K$.
\begin{defn}\label{def1}
 The differential field $F$  is said to be a {\it generalized elementary extension}\footnote{These are also called {\em liouvillian extensions}} of the differential field  $K$
if $K\subset F$ and there exists a chain of
differential fields $K=F_0\subset \dots\subseteq F_n\supset F$
such that $F_{i+1}=F_i<y_i>$ for every $i=0$, $\dots$, $n-1$ where $y_i$ is an exponential, a logarithm, or an algebraic element over $F_i$.

An element $a\in F$ is  {\it a generalized elementary element } over $K$, $K\subset F$, if it is contained in a certain generalized elementary elementary extension of the field $K$. \end{defn}

The following lemma is obvious.

\begin{lem}\label{lem2} An extension  $K\subset F$
is a {\it generalized elementary extension}
if and only if there exists a chain of
differential fields $K=F_0\subseteq \dots\subseteq F_n\supset F$
such that for every $i=0$, $\dots$, $n-1$, either
$F_{i+1}$ is a finite  extension of $F_i$, or $F_{i+1}$ is a pure transcendental extension of $F_i$  obtained by
adjoining  finitely many exponentials and logarithms over $F_i$.
\end{lem}

\paragraph{Extensions  by generalized quadratures}

Let $F\supset K$ be an  extension of a differential fields $K$.
\begin{defn}\label{def3}
 The differential field $F$  is said to be an {\it extension of the differential field  $K$  by generalized quadratures}
if $K\subset F$ and there exists a chain of
differential fields $K=F_0\subset \dots\subseteq F_n\supset F$
such that $F_{i+1}=F_i<y_i>$ for every $i=0$, $\dots$, $n-1$ where $y_i$ is an exponential of integral, an integral, or an algebraic element over $F_i$.
An element $a\in F$ is  {\em representable by generalized quadratures } over $K$, $K\subset F$, if it is contained in a certain generalized extension of the field $K$ by elementary generalized quadratures.
\end{defn}

\begin{defn}\label{def4}
An extension $F$ of a differential field $K$ is said to be:

1) a {\it generalized extension by integral}  if there are $y\in F$ and $f\in K$ such that $y'=f$, $y$ is transcendental  over $K$, and $F$ is a finite extension of the field $K\langle y\rangle$,

2) a {\it generalized extension by exponential of integral} if there are $y\in  F$, $f\in K$ such that $y'=fy$, $y$ is transcendental  over $K$, and  $F$ is a finite extension of the field $K\langle y\rangle$,
\end{defn}

The following lemma is obvious.

\begin{lem}\label{lem5} An extension  $K\subset F$
is an  extension by generalized quadratures
if  there is a chain  $K=F_0\subset \dots\subseteq F_n$
such that $F\subset F_n$ and for every $i=0$, $\dots$, $n-1$  or
$F_{i+1}$ is a finite extension of $F_i$, or
$F_{i+1}$ is  a generalized extension by integral of $F_i$, or
$F_{i+1}$ is  a generalized extension by exponential integral of $F_i$.
\end{lem}

\subsubsection{Functional differential fields and their extensions}
Let $K$ be a subfield in the field $F$ of all meromorphic functions on a connected domain $U$ of the Riemann sphere $\Bbb C^1\cup \infty$ with the fixed coordinate function $x$ on $\Bbb C^1$.
Suppose that $K$ contains all complex constants and is stable under differentiation
(i.e. if $f\in K$, then $f\,'=df/dx\in K$).
Then $K$ provides an example of a {\it functional differential field.}

Let us now give a general definition.
\begin{defn}\label{def6} Let $U,x$ be a pair consisting of a connected Riemann surface $U$ and a non constant meromorphic function $x$ on $U$. The map $f\rightarrow df/\pi^*dx$ defines the derivation  in   the field $F$ of all meromorphic
functions on  $U$ (the ratio of two meromorphic 1-forms is a well-defined meromorphic function).
 A {\it functional differential field} is any differential subfield of $F$
(containing all complex constants).
\end{defn}

The following construction helps to {\it extend} functional differential fields.
Let $K$ be a differential subfield of the field of meromorphic functions on a
connected Riemann surface $U$
equipped with a meromorphic function $x$.
Consider any connected Riemann surface $V$ together with a
nonconstant analytic map $\pi: V\to U$.
Fix the function $\pi^*x$ on $V$.
The differential field $F$ of all meromorphic functions on $V$ with the differentiation
$\varphi'= d\varphi/\pi^*dx$ contains the differential subfield $\pi^* K$ consisting of functions
of the form $\pi^* f$, where $f\in K$.
The differential field $\pi^*K$ is isomorphic to the differential field $K$,
and it lies in the differential field $F$.
For a suitable choice of the surface $V$,
an extension of the field $\pi^*K$, which is isomorphic to $K$, can be done within
the field $F$.

Suppose that we need to extend the field $K$, say, by an integral $y$ of some function $f\in K$.
This can be done in the following way.
Consider the covering of the Riemann surface $U$ by the Riemann surface $V$ of an
indefinite integral $y$ of the form $fdx$ on the surfave $U$.
By the very definition of the Riemann surface $V$,
there exists a natural projection $\pi: V\to U$, and
the function $y$ is a single-valued meromorphic function on the surface $V$.
The differential field $F$ of meromorphic functions on $V$ with the differentiation
$\varphi'= d \varphi/\pi^*dx$ contains the element $y$ as well as the field $\pi^*K$ isomorphic to $K$.
That is why the extension $\pi^*K\langle y\rangle$ is well defined as a subfield of
the differential field $F$.
We mean this particular construction of the extension whenever we talk about extensions
of functional differential fields.
The same construction allows to adjoin
a logarithm, an exponential, an integral or an exponential of integral
of any function $f$ from a functional differential field $K$ to $K$.
Similarly, for any functions $f_1,\dots,f_n\in K$, one can adjoin a
solution $y$ of an algebraic equation $y^n +f_1y^{n-1}+\dots+f_n=0$
or all the solutions
$y_1$, $\dots$, $y_n$ of this equation to $K$
(the adjunction of all the solutions $y_1$, $\dots$, $y_n$
can be implemented on the Riemann surface of the vector-function
$\bold y= y_1$, $\dots$, $y_n$).
In the same way, for any functions $f_1,\dots,f_{n+1}\in K$, one can adjoin
the $n$-dimensional $\Bbb C$-affine space of all solutions of the
linear differential equation
$y^{(n)}+f_1y^{(n-1)}+\dots +f_ny+f_{n+1}=0$ to $K$.
(Recall that a germ of any solution of this linear differential equation admits an
analytic continuation along a path on the surface $U$ not passing through the poles
of the functions $f_1$, $\dots$, $f_{n+1}$.)

Thus, {\it all above--mentioned extensions of functional differential fields can be implemented
without leaving the class of functional differential fields}.
When talking about extensions of functional differential fields, we always mean this particular
procedure.

The differential field of all complex constants and the differential field of
all rational functions of one variable can be regarded as differential
fields of functions defined on the Riemann sphere.

\subsubsection{Classes of functions and operations on multivalued functions}\label{classes}

An indefinite integral of an elementary function is a function rather than an element of an abstract differential field.
In functional spaces, for example, apart from differentiation and algebraic operations,
an absolutely non-algebraic operation is defined, namely, the composition.
Anyway, functional spaces provide more means for writing ``explicit formulas''
than abstract differential fields.
Besides, we should take into account that functions can be multivalued,
can have singularities and so on.

In functional spaces, it is not hard to formalize the problem of unsolvability
of equations in explicit form.
One can proceed as follows:
fix a class of functions and say that an equation is solvable explicitly if
its solution belongs to this class.
Different classes of functions correspond to different notions of solvability.

\paragraph{Defining classes of functions by the lists
of data}

A class of functions can be introduced by specifying a list of {\it basic functions}
and a list of {\it admissible operations}.
Given the two lists, the class of functions is defined as the set
of all functions that can be obtained
from the basic functions by repeated application of admissible operations.
Below, we define the class of
{\it generalized elementary functions} and the class of {\it generalized elementary functions over a functional differential field $K$} in exactly this way.

Classes of functions, which appear in the problems of integrability in finite terms,
contain multivalued functions.
Thus the basic terminology should be made clear.
We work with multivalued functions ``globally'', which leads
to a more general understanding of classes of functions defined by
lists of basic functions and of admissible operations.
A multivalued function is regarded as a single entity.
{\it Operations on multivalued functions} can be defined.
The result of such an operation is a set of multivalued functions;
every element of this set is called a function obtained from the given functions
by the given operation.
A  {\it class of functions} is defined as the set of all (multivalued) functions
that can be obtained from the basic functions by repeated
application of admissible operations.

\paragraph{Operations on multivalued functions}

Let us define, for example, the sum of two multivalued functions on a connected Riemann surface $U$.

\begin{defn}\label{def7}
Take an arbitrary point $a$ in $U$, any germ
$f_a$ of an analytic function $f$ at the point $a$ and any germ $g_a$ of an analytic function $g$
at the same point $a$.
We say that the multivalued function $\varphi$ on $U$ generated by the germ $\varphi_a=f_a+g_a$
{\it is representable as the sum of the functions} $f$ and $g$.
\end{defn}

For example, it is easy to see that exactly two functions of one variable
are representable in the form
$\sqrt{x}+\sqrt{x}$, namely, $f_1=2\sqrt{x}$ and $f_2\equiv 0$.
Other operations on multivalued functions are defined in exactly the same way.
{\it For a class of multivalued functions, being stable under addition means that,
together with any pair
of its functions, this class contains all functions representable as their sum.}
The same applies to all other operations on multivalued functions understood
in the same sense as above.

In the definition given above, not only the operation of addition plays a key role but
also the operation of analytic continuation hidden in the notion
of multivalued function.
Indeed, consider the following example.
Let $f_1$ be an analytic function defined on an open subset $V$ of the complex line
$\Bbb C^1$ and admitting no analytic continuation outside of $V$,
and let $f_2$ be an analytic function on
$V$ given by the formula $f_2=-f_1$.
According to our definition, the zero function is representable in the form
$f_1+f_2$ {\it on the entire complex line}.
By the commonly accepted viewpoint, the equality $f_1+f_2=0$ holds inside
the region $V$ but not outside.

Working with multivalued functions globally, we do not insist on the existence of
{\it a common region}, were all necessary operations would be performed on
single-valued branches of multivalued functions.
A first operation can be performed in a first region,
then a second operation can be performed in a second, different region
on analytic continuations of functions obtained on the first step.
In essence, this more general understanding of operations is equivalent to including
analytic continuation to the list of admissible operations on the analytic germs.
\subsubsection{Generalized elementary functions}\label{elemclasses}

In this section we define the generalized elementary functions of one complex variable and the generalized elementary functions over a functional differential field. We also discuss a relation of these notions with generalized elementary extensions of differential fields. First we'll present needed lists of basic functions and of admissible operations.\\[0.1in]
\underline{List of basic elementary functions}
\begin{enumerate}
\item All complex constants and an independent variable $x$.
\item The exponential, the logarithm, and the power $x^\alpha$ where $\alpha$ is any constant.
\item The trigonometric functions sine, cosine, tangent, cotangent.
\item The inverse trigonometric functions arcsine, arccosine, arctangent, arccotangent.
\end{enumerate}

\begin{lem}\label{lem8}
Basic elementary functions can be expressed through the exponentials and
the logarithms with the help of complex constants, arithmetic operations
and compositions.
\end{lem}
Lemma~\ref{lem8}  can be considered as a  simple exercise. Its proof can be found in \cite{[6]}.\\

\noindent \underline{List of some classical  operations}
\begin{enumerate}
\item The operation of composition takes functions $f$,$g$ to the function $f\circ g$.
\item The arithmetic operations take functions $f$, $g$ to the functions $f+g$, $f-g$, $fg$, and $f/g$.
\item The operation of differentiation takes function $f$ to the function $f'$.
\item The operation of integration takes function $f$ to a solution of equation $y'=f$ (the function $y$ is defined up to an additive constant).
\item The operation of taking  exponential of integral takes function $f$ to a solution of equation $y'=fy$ (the function $y$ is defined up to a multiplicative constant).
\item The operation of solving algebraic equations takes functions $f_1,\dots,f_n$ to the function $y$ such that $y^n+f_1y^{n-1}+\dots+f_n=0$ ( the function $y$ is not quite uniquely determined by functions $f_1,\dots,f_n$ since an algebraic equation of degree $n$ can have $n$ solutions).
\end{enumerate}

\begin{defn}\label{def9} The class of {\it generalized elementary functions of one variable} is defined by the following data:

List of basic functions: basic elementary functions.

List of admissible operations: Compositions, Arithmetic operations,  Differentiation, Operation of solving algebraic equations.
\end{defn}

\begin{thm}\label{thm10}
A (possibly multivalued) function of one complex variable  belongs to the class of generalized elementary functions if and only if it belongs to some generalized elementary extension of the differential field of all rational functions of one variable.
\end{thm}

Theorem~\ref{thm10} follows from Lemma~\ref{lem8} (all needed  arguments can be found in \cite{[6]}).

Let $K$ be  a functional differential field  consisting of meromorphic { functions} on a connected Riemann surface $U$ equipped with a meromorphic function $x$.

\begin{defn}\label{def11} Class of {\em generalized elementary functions over a functional differential field $K$} is defined by the following data.

List of basic functions: all functions from the field $K$.

List of admissible operations: Operation of composition with a generalized elementary function $\phi$ that takes $f$ to $\phi\circ f$, Arithmetic operations,  Differentiation, Operation of solving algebraic equations.
\end{defn}

\begin{thm}\label{thm12}
A (possibly multivalued) function on the Riemann surface $U$ belongs to the class of generalized elementary functions over a functional differential field $K$ if and only if it belongs to some generalize elementary extension of  $K$.
\end{thm}

Theorem~\ref{thm12} follows from Lemma~\ref{lem8} (all needed  arguments can be found in \cite{[6]}).

\subsubsection{Functions representable by generalized quadratures}

Here we define  functions of one complex variable representable by generalized quadratures and functions representable by generalized quadratures  over a functional differential field. We also discuss a relation of these notions with extensions of functional differential fields by generalized quadratures. First we'll present needed lists of basic functions and of admissible operations.

\begin{defn}\label{def13} The class of functions of one complex variable {\it representable by generalized quadratures} is defined by the following data:

List of basic functions: basic elementary functions.

List of admissible operations: Compositions, Arithmetic operations,  Differentiation, Integration,Operation of taking exponential of integral, Operation of solving algebraic equations.
\end{defn}
\begin{thm}\label{thm14}
A (possibly multivalued) function of one complex variable  belongs to the class of functions representable by generalized quadratures  if and only if it belongs to some extension   of the differential field of all constant functions of one variable by generalized quadratures.
\end{thm}

Theorem ~\ref{thm14} follows from Lemma~\ref{lem8} (All needed  arguments can be found in \cite{[6]}).

Let $K$ be  a functional differential field  consisting of meromorphic {functions} on a connected Riemann surface $U$ equipped with a meromorphic function $x$.

\begin{defn}\label{def15} The class of functions representable by {\em generalized quadratures over the functional differential field $K$} is defined by the following data:

List of basic functions: all functions from the field $K$.

List of admissible operations: Operation of composition with a generalized elementary function $\phi$ that takes $f$ to $\phi\circ f$, Arithmetic operations,  Differentiation, Integration, Operation of taking exponential of integral, Operation of solving algebraic equations.
\end{defn}

\begin{thm}\label{thm16}
A (possibly multivalued) function on the Riemann surface $U$ belongs to the class of generalized quadratures over a functional differential field $K$ if and only if it belongs to some  extension of  $K$ by generalized quadratures.
\end{thm}

Theorem~\ref{thm16} follows from Lemma 8 (all needed  arguments can be found in \cite{[6]}).

\subsection{Liouville's First Theorem and Actions of Lie Groups}\label{sec:first}
\subsubsection{ Introduction}

In 1833 Joseph Liouville proved the following fundamental result.\\

\noindent  {\sc Liouville's First Theorem}  {\it An  integral $y$ of a function $f$ from a functional differential field $K$ is a generalized elementary function over $K$
if and only if $y$ is representable in the form
\begin{eqnarray}
y(x)= \int\limits_{x_0}^x f(t)\,d t=r_0(x)+\sum\limits_{i=
1}^m\lambda_i\ln r_i(x), \label{liou1}
\end{eqnarray}
where $r_0,\dots,r_m\in K$  and $\lambda_1,\dots,\lambda_m$ are complex constants.}

For large classes of functions algorithms based on  Liouville's Theorem make it possible to either evaluate an integral or to prove that the integral cannot be ``evaluated in finite terms".

In this section we prove the Liouville's First Theorem. We follow the presentation from the paper \cite{[8]}.

Let  $K(y)\supset K$ be an extension  obtained by adjoining to a functional differential field $K$ an integral $y$ over  $K$.
The differential Galois group $G$ of this extension  does not contain enough information to determine if the integral $y$ belongs to a generalized  elementary extension of $K$ or not. Indeed, if the integral $y$ does not belong to $K$ then  group $G$ is always the same: it  is isomorphic to the additive group of complex numbers. From this fact  one can conclude that the Galois theory is not sensitive enough for proving  Liouville's First Theorem.
Nevertheless,  Liouville's First Theorem can be proved using differential Galois groups.

The first step  towards such a proof was suggested by Abel (see sections~\ref{sec2.3.1} and \ref{sec2.3.2}). This step is related to algebraic extensions and their finite Galois groups.

 A second step (see section~\ref{sec2.4})  deals with a pure transcendental extension $F$ of a functional differential field $K$, obtained by adjoining $k+n$  logarithms  and exponentials,  algebraically  independent over $K$. The differential Galois group of the extension $K\subset F$  is an $(k+n)$-dimensional connected commutative algebraic group $G$. It has a natural represented as a group of analytic automorphisms of an analytic variety $X$. Thus $G$ acts not only on the differential field $F$  but also on other objects such as closed 1-forms on $X$. This action plays a key role in our proof (see section~\ref{sec2.4.1}).

\bigskip

\subsubsection{Outline of an inductive proof}
Let us  outline an inductive proof of  Liouville's Theorem.
\begin{defn}\label{def17} A function $g$ is a {\it generalized elementary function of complexity $\leq k$} if there is a chain $K=F_0\subset F_1\subset\ldots\subset F_k$ of functional differential fields such that $g\in F_k$ and for any $0\leq i<k$ either $F_{i+1}$ is a  finite  extension of  $F_{i}$, or $F_{i+1}$ is a pure transcendental extension of $F_i$ obtained by adjoining finitely many  exponentials, and  logarithms over $F_{i}$.
\end{defn}

We will prove the following induction hypothesis $I(m)$: {\it  Liouville's Theorem is true for every integral $y$ of complexity  $\leq m$ over any functional
differential field $K$}.
The statement $I(0)$ is obvious: if  $y\in K$, then $y=r_0\in K$. Now let  $y\,'\in K$ and $y\in F_k$. Since $y\,'\in F_1$,  by induction  $y=R_0 + \sum_{i=1}^q\lambda_i\ln R_i,$
where $R_0$, $R_1$, $\ldots$, $R_q \in F_1$. We need to show that $y$ is representable in the form (\ref{liou1}) with $r_0,\ldots,r_m\in F_0=K$ We have the following two cases to consider:

1. $F_1$  is  a finite  extension of $F_0=K$. The statement of induction hypothesis in that case was proved by Abel and is called the Abel's Theorem. We will present its proof in  section~\ref{sec2.3}.

2. $F_1$  a pure transcendental extension of $F_0=K$  obtained by adjoining exponentials and logarithms  over  $K$. We will deal with this case in section~\ref{sec2.4}.

\subsubsection{Algebraic case}\label{sec2.3}

In  Subsection~\ref{sec2.3.1} we  discuss finite extensions of differential fields. In  subsection~\ref{sec2.3.2} we  present a proof of the Abel's Theorem.

\paragraph{An algebraic extension of a functional differential field}\label{sec2.3.1}
\
Let
\begin{eqnarray}
P(z)=z^n+a_1z^{n-1}+\dots+a_n \\abel{eqn2}
\end{eqnarray}
be an
irreducible polynomial over  $K$,
$P\in K[z]$.
Suppose that a functional differential field $F$ contains $K$ and a root $z$ of $P$.

\begin{lem}\label{lem18}
 The field $K(z)$ is stable under the differentiation.
\end{lem}
\begin{proof}Since  $P$ is irreducible over $K$, the polynomial
  $\frac{\partial P}{\partial z}$ has no common roots with $P$ and is different from zero
  in the field $K[z]/(P)$. Let $M$ be a polynomial  satisfying a congruence $M\frac{\partial P}{\partial z}\equiv -
  \frac{\partial P}{\partial x}\pmod P$.
  Differentiating the identity $P(z)=0$ in the field $F$,
we obtain that $\frac{\partial P}{\partial z}(z)z'+
\frac{\partial P}{\partial x}(z)=0$, which implies that $z'=M(z)$.
Thus the derivative of the element $z$ coincides with the value at $z$ of
a polynomial $M$. Lemma~\ref{lem18} follows from this fact.
\end{proof}

Let $K\subset F$ and $\hat K\subset \hat F$ be  functional differential fields, and
$P$, $\hat P$
irreducible polynomials over $K$,
$\hat K$ correspondingly.
Suppose that  $F$, $\hat F$ contain  roots $z$, $\hat z$
of  $P$, $\hat P$.

\begin{thm}\label{thm19} Assume that there is an isomorphism  $  \tau:K\rightarrow \hat K$ of differential fields $K$, $\hat K$ which maps coefficients of the polynomial $P$ to the corresponding coefficients of the polynomial $\hat P$. Then $\tau$ can be extended in a unique way to the differential isomorphism $ \rho:K(z)\rightarrow \hat K(\hat z)$.
\end{thm}
Theorem~\ref{thm19}  could be obtained by  the  arguments  used in the proof of Lemma~\ref{lem18}.

\paragraph{Induction hypothesis for an algebraic extension}\label{sec2.3.2}

Let $z_1,\dots,z_n$ be the roots of the polynomial  $P$ given by (2) and let $F_1=K\langle z_1 \rangle$.  Assume that there is an element $y_1\in F_1$, such that  $y_1'\in K$, $M_i\in K[x]$ and $y_1'$ is representable in the form
\begin{eqnarray}
y_1'=\sum_{i=1}^q\lambda_i\frac{(M_i(z_1))'}{M_i(z_1)}+(M_0(z_1))'. \label{eq3}
\end{eqnarray}

\noindent  {\bf Abel's Theorem} Under the above assumptions the element $y_1'$ is representable in the form (1) with polynomials  $M_i$ independent of $z_1$, i.e. with $M_0,M_1,\dots,M_q\in K$.

\begin{proof} Let $y_1$ be equal to $Q(z_1)$ where $Q\in K[z]$. For any $1\leq j\leq n$ let $y_j$ be the element $Q(z_j)$. According to Theorem 19 the identity (\ref{eq3}) implies the identity

\begin{eqnarray}
y_j'=\sum_{i=1}^q\lambda_i\frac{(M_i(z_j))'}{M_j(z_1)}+(M_0(z_j))'. \label{eq4}
\end{eqnarray}

Since $y'_1\in K$ we obtain $n$ equalities  $y'_1=\dots=y'_n$. To complete the proof it is enough to
take the arithmetic mean of  $n$ equalities (4). Indeed  the elements
$\tilde M_i=\prod_{1\leq k\leq n}M_i(z_k)$ and $\tilde M_0=\sum_{1\leq k\leq n}M_0(z_k)$ are symmetric functions in the roots of the polynomial $P$ thus $\tilde M_0, \dots \tilde M_q\in K$.
\end{proof}

\begin{remark}
The proof uses implicitly the Galois group $G$ of the  splitting field of the polynomial $P$ over the field $K$. The group $G$ permutes the roots $y_1,\dots,y_n$ of $P$. The  element $\tilde M_i=\prod_{1\leq k\leq n}M_i(z_k)$ and $\tilde M_0=\sum_{1\leq k\leq n}M_0(z_k)$ are invariant under the action of $G$  thus they belong to the field $K$.
\end{remark}

\subsubsection{Pure transcendental  case}\label{sec2.4}

Here we prove the induction hypothesis in the pure transcendental case. First we will state the corresponding  Theorem~\ref{thm21} and will outline its proof.

Let $F_1$ be a functional differential field obtained by  extension of the functional differential field $K$   by adjoining algebraically independent over $K$ functions
\begin{eqnarray}
y_1=\ln a_1, \dots, y_k=\ln a_k, z_1=\exp b_1,\dots,z_n=\exp b_n \label{eq5}
\end{eqnarray}
where $a_1,$ $\dots,$ $a_k,$ $b _1,$ $\dots,$ $b_k$ are some functions from $K$. We will assume that $F_1$ consists of meromorphic functions on a connected Riemann surface $U$ and the differentiation in $K_1$ using a meromorphic function $x$ on $U$. Let $X$ be the manifold $U\times G$ where $G=\Bbb C^k\times (\Bbb C^*)^n$. Consider a map $\gamma:U\rightarrow \Bbb C^k\times (\Bbb C^*)^n$ given by formula
\begin{eqnarray}
\gamma(p)=y_1(p),\dots, y_k(p),z_1(p)\dots,z_n(p)
\end{eqnarray}
where the functions $y_i$, $z_j$ are defined by (\ref{eq5}).

Let $X$ be the product $U\times (\Bbb C)^k\times (\Bbb C^*)^n$. Denote by $\Gamma\subset X$ the graph of the map~$\gamma$. Consider a germ $\Phi$ of a complex valued function at the point $a\in X$.

\begin{defn}\label{def20}
We say that  $\Phi$ is a {\it logarithmic type germ} if $\Phi$ is representable in the form $\Phi_a=R_0+\sum_{i=1}^q\lambda_i\ln R_i,$
where $R_i$ are germs at the point $a\in X$ of rational functions of $(y_1,\dots,y_k, z_1,\dots,z_n)$ with coefficients in $K$ and $\lambda_j$ are complex numbers.
\end{defn}

\begin{thm}\label{thm21}  Let $\Phi$ be a logarithmic type germ at a point $a=(p_0,\gamma (p_0))\in \Gamma$. Then the germ of the function $\Phi (p,\gamma(p))$ at the point $p_0\in U$  is a germ of an  integral over $K$ if and only if $\Phi$ is representable in the following form
\begin{eqnarray}\Phi(p,y,z)=\Phi(p, \gamma(p_0))+\sum_{i=1}^k c_i (y_i -y_i(p_0))+\sum_{j=1}^nt_j\ln \frac{z_j}{z_j(p_0)}\label{eq6} \end{eqnarray}
where $r_0$ is a germ of a function from the field $K$ and $c_i,t_j$ are complex constants.
\end{thm}
Theorem ~\ref{thm21} proves induction hypothesis in the pure transcendental case. Indeed the germ $\Phi(p,\gamma(p_0))$ given by (\ref{eq6}) is a germ of a function from the field $K$ and according to (5) the  identities $c_iy_i=c_i\ln a_i$, $t_j\ln z_j=t_jb_j$ hold. We split the claim of Theorem 7 into two parts.

First we consider the natural action of the group $G=(\Bbb C^k)\times (\Bbb C^*)^n$ on $X=U\times G$ and we  describe all germs of closed 1-forms locally invariant under this action.  Corollary~\ref{cor25} claims that each such  1-form is a differential of a function representable in the form (\ref{eq6}).

Second we show that if the germ $\Phi$ satisfies the conditions of Theorem~\ref{thm21} then the germ $d\Phi$ is locally invariant under the action of the group $G$ (see Theorem~\ref{thm31}).

\paragraph{Locally invariant closed 1-forms}\label{sec2.4.1}

Let $G$ be a connected Lie group acting by diffeomorphisms on a manifold $X$. Let $\pi:G\rightarrow Diff (X)$ be a corresponding homomorphism from $G$ to the group $Diff (X)$ of diffeomorphisms of $X$. For a vector $\xi$ from the Lie algebra $\mathcal{G}$ of $G$ the action $\pi$ associates the vector field $V_\xi$ on $X$. The germ $\omega_{x_0}$ at a point $x_0\in X$ of a differential form $\omega$ on $X$ is {\it locally invariant under the action $\pi$} if for any $\xi\in \mathcal{G}$ the Lie derivative $L_{V_\xi}\omega$ is equal to zero.

\begin{lem}\label{lem22} The germ of the differential $d \varphi_{x_0}=\omega_{x_0}$ of a smooth function $\varphi$  is locally invariant under the action $\pi$ if and only if for each  $\xi\in\mathcal{G}$ the Lie derivative $L_{V_\xi}\varphi$ is a constant $M(\xi)$ (which depends on $\xi$).
\end{lem}
\begin{proof}  Applying ``Cartan's magic formula"
$
L_{V_\xi}\omega =i_{V_\xi}d\omega +d(i_{V_\xi}\omega )
$
we obtain that $L_{V_\xi}\omega=0$ if and only if $d(L_{V_\xi}\varphi)=0 $ which means that  $L_{V_\xi}\varphi$ is constant.
\end{proof}

The following theorem characterizes locally invariant  closed 1-forms  more explicitly.

\begin{thm}\label{thm23}  The germ of the differential $d \varphi_{x_0}=\omega_{x_0}$ of a smooth complex valued function $\varphi$  is locally invariant under the action $\pi$ if and only if there exists a local homomorphism $\rho$ of $G$ to the additive group  $\Bbb C$ of complex numbers such that for any $g\in G$ in a neighborhood of the identity  the following relation  holds:
$$
\varphi(\pi(g)x_0)=\varphi(x_0)+\rho(g).
$$
\end{thm}

\begin{proof} For $\xi\in \mathcal G$  the Lie derivative $L_{V_\xi}\varphi$ is constant $M(\xi)$  by Lemma 8. Let us show that for $\xi\in [\mathcal G ]$ where $[\mathcal G ]$ is the commutator  of $\mathcal G$ the constant $M(\xi)$ equals to zero. Indeed if $\xi=[\tau,\rho]$ then
$$
L_{V_\xi} \varphi= L_{V_\rho}L_{V_\tau}\varphi-L_{V_\tau}L_{V_\rho}\varphi=L_{V_\rho}M(\tau)-L_{V_\tau}M(\rho)=0.
$$
Thus the linear function $M:\mathcal G \rightarrow \Bbb C$  mapping $\xi$ to $M(\xi)$ provides a homomorphism of $\mathcal G $ to the Lie algebra of the additive group $\Bbb C$ of complex numbers. Let $\rho$ be the local homomorphism of $G$ to $\Bbb C$ corresponding to the homomorphism $M$.

Consider a function $\phi$  on a neighborhood of the identity  in $G$ defined by the following formula: $\phi(g)=\varphi(x_0)+\rho(g)$. By definition on a neighborhood of identity the function $\phi$ has the same differential as the function $\varphi(\pi(g)x)$. Values  of these functions at the identity  are equal to $\varphi(x_0)$. Thus these functions are equal.
\end{proof}

Assume that  $X=U\times G$ where $U$ is a  manifold  and an action $\pi$ is given by the formula $\pi(g) (x,g_1)= (x, gg_1)$. Applying Theorem ~\ref{thm23} to this action  we obtain the following corollary.

\begin{cor}\label{cor24}  If  germ of  differential $d \varphi=\omega$ of a smooth complex valued function $\varphi$ at a point $(x_0,g_0)\in U\times G$ is locally invariant under the action $\pi$ then  in a neighborhood of the point $(x_0,g_0)$ the following identity holds:
\begin{eqnarray}
\varphi(x, g)=\varphi(x, g_0)+\rho(gg_0^{-1}).\label{eq7}
\end{eqnarray}
where $\rho$ is a local homomorphism of $G$ to the additive group of complex numbers.
\end{cor}

\begin{proof} Follows from Theorem 9 since the element $gg_0^{-1}$ maps the point $(x,g_0)$ to the point $(x,g)$.
\end{proof}

Let $G$ be the group $\Bbb C^k\times (\Bbb C^*)^n$ where $\Bbb C$ and $\Bbb C^*$ are additive and duplicative group of complex numbers. We will consider the group $\Bbb C^k\times (\Bbb C^*)^n$ with coordinate functions   $(y,z)=(y_1,\dots, y_k,z_1,\dots, z_n)$ assuming that $z_1\cdot\dots\cdot z_n\neq 0$.

\begin{cor}\label{cor25}If in the assumptions of Corollary~\ref{cor24} for $G=\Bbb C^k\times (\Bbb C^*)^n$ in a neighborhood of $(x_0,y_0,z_0)\in U\times(\Bbb C^k\times (\Bbb C^*)^n )$ the following identity holds
$$
\varphi(x,y,z)=\varphi(x,y_0,z_0)+\sum_{1\leq i\leq k} \lambda_i(y_i-(y_0)_{i}) +\sum_{1\leq j \leq n} \mu_j\ln \frac{z_j}{(z_0)_{j}}
$$
where $\lambda_1,\dots,\lambda_k, \mu_1\dots,\mu_n$ are complex constants.
\end{cor}

\begin{proof} Follows from (\ref{eq7}) since any local homomorphism $\rho$ from the group $\Bbb C^k\times (\Bbb C^*)^n$ to the additive group of complex numbers can be given by formula
 $$
 \rho(y_1,\dots,y_k,z_1,\dots,z_n)=\sum _{1\leq i \leq k}\lambda_1 y_i+\sum_{1\leq j \leq n}\mu_j\ln z_j
 $$
 where $\lambda_i$ and $\mu_j$ are complex constants.
\end{proof}

\paragraph{Vector field associated to a logarithmic-exponential extension}

We use the notation introduced  in   subsection~\ref{sec2.4}. Let $G$ be the group $\Bbb C^k\times (\Bbb C^*)^n$ and let $X$ be the product $U\times G$ consider the map $\gamma:U\rightarrow  \Bbb C^k\times (\Bbb C^*)^n$ given by the following formula:

\begin{eqnarray}
y_1=\ln a_1,\dots,y_k=\ln a_k, z_1=\exp b_1,\dots,z_n=\exp b_n. \label{eq8}
\end{eqnarray}

The map $\gamma$ satisfies  the following differential relation:
$$
  d\gamma= da_1/a_1,\dots,da_k/a_k,z_1d b_1,\dots,z_ndb_n
$$
\begin{defn}\label{def26} Let $V$ be a meromorphic vector field on $X$ defined by the following conditions. If $V_a$ is the value of $V$ at the point $a=(p,y_1,\dots, y_k,z_1,\dots,z_n)\in X$ then $\langle dx, V_a\rangle =1$, $\langle dy_i, V_a\rangle =a_i'/a_i(p)$ for $1\leq i\leq k$, $\langle dz_j, V_a\rangle =b_j'/b_i(p)$ for $1\leq j\leq n$
\end{defn}

The vector field $V$ is regular on $U^0\times G$ where $U^0$ is an open subset in $U$ which does not contain the zeroes and poles of the functions $a_1,\dots, a_k$ and poles of functions $b_1,\dots,b_n$ poles and zeros and poles of the 1-form $dx$. By construction the graph $\Gamma=(p,\gamma(p)\subset X$ of the map $\gamma$ is an integral curve for the differential equation on $X$ defined by the vector field $V$.

The following lemmas are obvious.

\begin{lem}\label{lem27} The vector field $V$  is invariant under the action $\pi$ on $X$. For each element $g\in G$ the curve  $g \Gamma\subset X$ of the graph $\Gamma$ of $\gamma$ is an integral curve for $V$.
\end{lem}

\begin{lem}\label{lem28}  The field $K(y,z)$ of rational functions in  $y_1,\dots,y_k,z_1,\dots, z_n $ over the field $K$ is invariant under the action $\pi$ on $X$. For each vector $\xi\in \mathcal G$ in the Lie algebra $\mathcal G$ of $G$ the Lie derivative $L_{V_\xi} R$ of $R\in K(y,z)$ belongs to $K(y,z)$.
\end{lem}

\paragraph{Pure transcendental logarithmic exponential extension}

We will assume below that the components (\ref{eq8}) of $\gamma$ are algebraically independent over $K$.\\

\noindent {\bf Liouville's principal.} If a polynomial $P\in K[y_1,\dots,y_k,z_1,\dots,z_n]$ vanishes on the graph $\Gamma\subset X$ of the map $\gamma$ then $P$ is identically equal to zero.

\begin{proof} If $P$ is not identically equal to zero then the components of $\gamma$ are algebraically dependent over the field $K$.
\end{proof}

\begin{thm}\label{thm29} The extension $K\subset F_1$ is isomorphic to the extension of $K$ by the field of rational functions $K(y,z)$ in $(y_1,\dots,y_k,z_1,\dots,z_n)$ over $K$ considered as the field of functions on $X$  equipped with the differentiation sending $f\in K(y,z)$ to the Lie derivative $L_Vf$ with respect to the vector field $V$ introduced in definition 8.
\end{thm}

\begin{proof}By assumption components (\ref{eq8}) of the map $\gamma$ are algebraically independent over $K$ thus each function from the extension obtained by adjoining to  $K$  these components is representable in the unique way as a rational function from $K(y,z)$. By definition the derivatives of the components (\ref{eq8})  coincide with their Lie derivatives  with respect to the vector field $V$.
\end{proof}

The action $\pi$ of the group $G=\Bbb C^k\times(\Bbb C^*)^n$ on $X$ induces the action $\pi^*$ of $G$ on the space of functions on $X$ containing the field  $K(y,z)$. The vector field $V$ is invariant under the action $\pi$. Thus $\pi^*$ acts on $K(y,z)\sim F_1$ by  differential automorphisms. It is easy to see that a function $f\in K(y,z)$ is fixed under the action $\pi^*$   if and only if  $f\in K$, i.e. the group $G$ is isomorphic to the {\it differential Galois group} of the extension $K\subset F_1$. We proved the following result

\begin{thm}\label{thm30} The differential Galois group of the extension $K\subset F_1$ is isomorphic to the group $G$. The Galois group  is induced on the differential field $K(y,z)$ with the differentiation given by Lie derivative with respect to the field $V$   by the action of $G$ on the manifold $X=U\times \Bbb C^k\times (\Bbb C^*)^n$.
\end{thm}

Now we are ready to complete inductive proof of  Liouvile's First Theorem.

\begin{thm}\label{thm31} Let $\Phi$ be a logarithmic type germ at a point $a=(p_0,\gamma (p_0))\in \Gamma\subset X$. If the germ of the function $\Phi (p,\gamma(p))$ on $U$ at the point $p_0\in U$  is a germ of an  integral $f$ over $K$ then the germ of the differential $d\Phi$ at the point $a\in X$ is locally invariant under the action $\pi$ on $X$.
\end{thm}

\begin{proof}  By the assumption of Theorem the restriction of the function $(L_V\Phi-f)$  on $\Gamma$  is equal to zero. Since the function $(L_V\Phi-f)$ belongs to the field $K(y,z)$  the function $(L_V\Phi-f)$ by  Liouville's principle is equal to zero identically on $X$. In particular it is equal to zero  on the integral curve $g \Gamma$ the vector field  $V$, where $g$ is an element of the group $G$. Thus  the restrictions of function  $L_V\pi(g)^*(\Phi-f)$ to $\Gamma$ equals to zero. Since the function $f$ is invariant under the action $\pi^*$ we obtain that the restriction on $\Gamma$ of $L_V(\Phi-\pi^*(g)\Phi)$ is equal to zero. Differentiating this identity we obtain that for any $\xi\in \mathcal G$ the restriction on $\Gamma$ of $L_V(L_{V_\xi}\Phi)$ equals to zero.
Thus on $\Gamma$ the function $L_{V_\xi}\Phi$ is constant.  Lemma 27 implies  that the function $L_{V_\xi}\Phi$ belongs to the field $K(x,y)$. Thus  the function $L_{V_\xi}\Phi$ is a constant on $X$ by Liouville's principal. Thus the 1-form $d\Phi$ is locally invariant under the action $\pi$  by Lemma 22. Theorem 31 is proved.
\end{proof}

Thus we complete proof of Theorem 21 and  the inductive proof of Liouville's First Theorem.
\bigskip


\medskip

\subsection{Liouville's Second Theorem and its Generalizations}\label{sec:second}

\subsubsection{Introduction}\label{secondintro}

In 1839 Joseph Liouville proved the following fundamental result.\\

\noindent  {\bf Liouville's Second  Theorem}
A second order homogeneous linear differential equation
\begin{eqnarray}
 y''+a_1y' +a_ny=0 \label{eq9}
\end{eqnarray}
 whose coefficients  $a_1,a_2$ belong to a functional differential field $K$ is solvable by generalized quadratures over $K$ if and only if it has a solution of the  form $y_1=\exp z$ where $z'$ is  algebraic  over $K$.

Much later this theorem was generalized for homogeneous linear differential equations of any order. Consider an  equation
\begin{eqnarray}
 y^{(n)}+a_1y^{(n-1)}+\dots +a_ny=0  \label{eq10}
\end{eqnarray}
 whose coefficients  $a_i$ belong to  $K$.

\begin{thm}\label{thm32}
If the equation~(\ref{eq10})  has a solution representable by generalized quadratures over $K$ then it necessarily has a solution of the  form $y_1=\exp z$ where $z'$ is  algebraic  over $K$.
\end{thm}

The following lemma is obvious.

\begin{lem}\label{lem33} Assume that the equation~(\ref{eq10}) has a solution $y_1$ representable by generalized quadratures  over $K$. Then
the equation~(\ref{eq10})  can be solved by generalized quadratures over $K$ if and only if the linear differential equation of order $(n-1)$ over the differential field $K(y_1)$  obtained from~(10) by the reduction of order using the solution $y_1$ is solvable by generalized quadratures over  $K(y_1)$.
\end{lem}

Indeed on one hand each solution of the equation obtained from (\ref{eq10}) by the reduction of order using $y_1$  can be expressed  in the form $(y/y_1)'$ where $y$ is a solution of  (\ref{eq10}). On the other hand any solution $y$ of the equation $(y/y_1)'=u$, where $u$ is represented by generalized quadratures over $K(y_1)$, is representable by generalized quadratures over  $K$, assuming that $y_1$ representable by generalized quadratures  over $K$.

Thus Theorem~\ref{thm32} provides the following criterium for solvability of the equation~(\ref{eq10}) by generalized quadratures.

\begin{thm}\label{thm35}
The equation~(\ref{eq10})  is solvable by generalized quadratures over $K$ if and only if the following conditions hold:

1) the equation~(\ref{eq10}) has a solution $y_1$ of the form $y_1=\exp  z$ where $z'=f$ is  algebraic over $K$,

2) the linear differential equation of order $(n-1)$ over  $K(y_1)$  obtained from~(\ref{eq10}) by the reduction of order using the solution $y_1$ is solvable by generalized quadratures over  $K(y_1)$.
\end{thm}

For $n=2$ Theorem~\ref{thm35} is equivalent to Liouville's Second Theorem because linear differential equations of first order are automatically solvable by quadratures.

The standard proof (E. Picard and E.~Vessiot, 1910) of Theorem~\ref{thm32}    uses the differential Galois theory and is rather involved (see \cite{[10]}).
\bigskip

In the case when the equation~(\ref{eq10}) is a Fuchsian differential equation and $K$ is the field of rational function of one complex variable Theorem~\ref{thm35} has a topological explanation (see section~\ref{ch2sec2} and \cite{[6]}) which allows to prove much stronger version of this result). But in general case Theorem~\ref{thm35} does not have a similar visual explanation.
\bigskip

Maxwell Rosenlicht in 1973 proved \cite{[13]}  the following theorem.

\begin{thm}\label{thm36}
Let $n$ be a positive integer, and let $Q$ be a polynomial in several variables with coefficients in a differential field $K$ and of total degree less than $n$. Then if the  equation

\begin{eqnarray}u^n=Q(u, u', u'', \dots) \label{eq11}
\end{eqnarray}
has a solution representable by generalized quadratures over $K$, it has a solution algebraic over $K$.
\end{thm}
The logarithmic derivative $u=y'/y$ of any solution of the equation (\ref{eq10}) satisfies the {\it generalized Riccati equation of order $n-1$ associated with (\ref{eq10})}, which is a particular case of the equation (\ref{eq11}). Rosenlicht showed that Theorem~\ref{thm32} easily follows from Theorem ~\ref{thm35} applied to the corresponding generalized Riccati equation  (see section~\ref{sec3.2}).
 In modern differential algebra abstract fields equipped with an operation of differentiation are considered. The Rosenlicht's proof of Theorem~\ref{thm35}  is not elementary: it is applicable to abstract differential fields of characteristic zero and  makes use of the valuation theory\footnote{ According to Michael Singer the valuation theory used in Rosenlicht's  proof  is  a fancy way of using power series methods (private communication). }.

The logarithmic derivative $u=y'/y$ of any solution $y$ of the homogeneous linear differential equation of second order~(\ref{eq9})  satisfies  the  Riccati equation
\begin{eqnarray}
u'+a_1u+a_2+u^2=0. \label{eq12}
\end{eqnarray}

To prove the Liouville's Second Theorem  Liouville and Ritt proved first Theorem~\ref{thm35} for the Riccati equation~(\ref{eq12}). To do that J.F.Ritt (in his simplification of the Liouville's proof) considered a special one parametric family of solutions of (\ref{eq12}) and  used  an expansion of these solutions  as  functions of the parameter  into converging  Puiseux series (see  Chapter V in \cite{[11]}). J.F.Ritt used a generalization of the following theorem based on ideas suggested by Newton.

Consider an algebraic function $z(y)$ defined by an equation  $P(y,z)=0$ where $P$ is a polynomial  with coefficients in a subfield $K$ of $\Bbb C$. Then all branches of the algebraic function $z(y)$ at the point $y=\infty$ can be developed into converging Puiseux series whose coefficients belong to a finite extension of the field $K$.

A {\it generalized Newton's Theorem} claims that the similar result holds if instead of a numerical field of coefficient one takes a field $K$ whose element are meromorphic functions on a connected Riemann surface. In  J.F.Ritt's book \cite{[11]} this result is proved in the same way as its classical  version using the Newton's polygon method.

J.F.Ritt's proof is written in  old mathematical language and  does not fit into  our presentation. Theorem~\ref{thm42} provides an exact statement of the generalized Newton's Theorem. It is presented  without proof: the  main arguments proving it are well known and classical. One also can obtain a proof modifying J.F.Ritt's exposition. Theorem~\ref{thm42}  plays a crucial  role in  section~\ref{sec3.3}. For  the sake of completeness I will present its modern proof in a separate  paper.
 \bigskip

In this section we  discuss a proof of Theorem~\ref{thm35} which does not rely on the valuation theory.  It generalizes J.F.Ritt's arguments (makes use of the Puiseux expansion via a generalized Newton's Theorem) and provides an elementary proof of the classical Theorem~\ref{thm32}. We follow the presentation from the paper \cite{[7]}.

\subsubsection{Generalized Riccati equation}\label{sec3.2}

Here we define the generalized Riccati equation and reduce Theorem~\ref{thm32} to Theorem~\ref{thm35}. In this section we also generalize Theorem~\ref{thm32} for nonlinear  homogeneous equations (this generalization will not be used in the next sections).

Assume that $u$ is the logarithmical derivative of a non identically equal to zero meromorphic function $y$, i.e the relation $y'=uy$ holds.
\begin{defn}\label{def36} Let $D_n$ be a polynomial in $u$ and in its derivatives $u,u',\dots,u^{(n-1)}$ up to order $(n-1)$ defined by induction by
the following conditions:
$$
D_0=1; \ \  D_{k+1}=\frac{d D_k}{d x}+u D_k.
$$
\end{defn}
\begin{lem}\label{lem37}  1) The polynomial $D_n$ has integral coefficients and $\deg D_n= n$. The degree $n$ homogeneous part of $D_n$ equals to $u^n$ (i.e. $D_n=u^n+\tilde D_n$ where $\deg \tilde D_n<n$).

2) If $y$ is a function whose logarithmic derivative equals to $u$ (i.e. if $y'=uy$) then for any $n\geq 0$ the relation $y^{(n)}=D_n(u) y$ holds.
\end{lem}
Both claims of Lemma~\ref{lem37} can be easily checked by induction.

Consider a homogeneous linear differential equation (\ref{eq10})
whose coefficients  $a_i$ belong to a differential field $K$.

\begin{defn}\label{def38}  The  equation
\begin{eqnarray}
 D_n+a_1D_{n-1}+\dots +a_nD_0=0 \label{eq13}
\end{eqnarray}
of  order $n-1$ is called the {\it generalized Riccati equation} for the homogeneous linear differential equation (\ref{eq10}).
\end{defn}

 \begin{lem}\label{lem39} A non identically equal to zero function $y$ satisfies the linear differential equation (\ref{eq10}) if and only if its logarithmic derivative $u=y'/y$ satisfies the generalized Riccati equation (\ref{eq13}).
\end{lem}
\begin{proof} Let $y$ be a nonzero solution of (\ref{eq10}) and let $u$ be its logarithmic derivative. Then dividing (\ref{eq10}) by $y$ and using the identity $y^{(k)}/y=D_k(u)$ we obtain that $u$ satisfies (\ref{eq13}). If $u$ is a solution of (\ref{eq13}) then multiplying (\ref{eq4}) by $y$ and using the identity $y^{(k)}=D_k(u)y$ we obtain that $y$ is a non zero solution of (\ref{eq10}).  \end{proof}

\begin{cor}\label{cor40} 1) The equation (\ref{eq10}) has a non zero solution  representable by generalized quadratures  over $K$ if and only if the equation (\ref{eq13}) has a solution  representable by generalized quadratures  over $K$.

2) The equation (\ref{eq10}) has a solution $y$ of the form $y=\exp z$ where $z'=f$ is an algebraic function over $K$ if and only if the equation (\ref{eq13}) has an algebraic solution over $K$.
\end{cor}

\begin{proof} 1) A non zero function $y$ is representable by generalized quadratures over $K$ if and only if  its logarithmic derivative $u=y'/y$ is representable by generalized quadratures over $K$.

2) A function $y$ is equal to $\exp z$ where $z'=f$ if and only if its logarithmic derivative is equal to $f$.
\end{proof}

The generalized Riccati equation (\ref{eq13}) satisfies the conditions of Theorem~\ref{thm35}. Thus Theorem~\ref{thm32} follows from Theorem~\ref{thm35} and from  Corollary~\ref{cor40}.

Let us generalize the results of this section. Consider an order $n$ homogeneous equation
\begin{eqnarray}
P(y,y',\dots,  y^{(n)})=0 \label{eq10'}
\end{eqnarray}
 where $P$ is a degree $m$ homogeneous polynomial in $n+1$ variables $x_0,x_1,\dots, x_{n}$ over a functional differential field $K$.
\begin{defn}\label{def38'} The  equation
\begin{eqnarray}
P(D_0,D_1, \dots, D_n)=0 \label{eq13'}
\end{eqnarray}
of  order $n-1$  is called the {\it generalized Riccati equation} for the  homogeneous equation (\ref{eq10'}).
\end{defn}

\begin{lem}\label{lem39'} A non identically equal to zero function $y$ satisfies the homogeneous  equation (\ref{eq10'}) if and only if its logarithmic derivative $u=y'/y$ satisfies the generalized Riccati equation (\ref{eq13'}).
\end{lem}

\begin{cor}\label{cor40'} 1) The equation (\ref{eq10'}) has a non zero solution  representable by generalized quadratures  over $K$ if and only if the equation (\ref{eq13'}) has a solution  representable by generalized quadratures  over $K$.

2) The equation (\ref{eq10'}) has a solution $y$ of the form $y=\exp z$ where $z'=f$ is an algebraic function over $K$ if and only if the equation (\ref{eq13'}) has an algebraic solution over $K$.
\end{cor}

Lemma~\ref{lem39'} and Corollary~\ref{cor40'} can be proved exactly in the way as Lemma~\ref{lem39} and Corollary~\ref{cor40}

Let us defined the {\em$\xi$-weighted degree $\deg_\xi x^p$} of the monomial  $x^p=x_0^{p_0}\cdot\dots\cdot x_n^{p_n}$ by the following formula: $$ \deg_\xi x^p=
\sum_{i=0}^{i=n} i p_i
.$$
We will say that a polynomial $P(x_0,\dots,x_n)$ satisfies the {\it $\xi$-weighted degree condition} if the sum of coefficients of all monomials in $P$ having the biggest $\xi$-weighted degree is not equal to zero. A polynomial $P$ having a unique  monomial with the biggest $\xi$-weighted degree  automatically satisfies this condition. For example a degree $m$ polynomial $P$ containing a term $ax_n^m$ with $a\neq 0$  automatically satisfies $\xi$-weighted degree condition.

\begin{thm}\label{thm32'} Consider the homogeneous equation~(\ref{eq10'}) with the polynomial $P$ satisfying the $\xi$-weighted degree condition. If this equation  has a solution representable by generalized quadratures over $K$ then it necessarily has a solution of the  form $y_1=\exp z$ where $z'$ is  algebraic  over $K$.
\end{thm}
\begin{proof} It is easy to check that if the polynomial $P$ satisfies the $\xi$-weighted degree condition then the generalized Riccati equation (\ref{eq13'}) satisfies the conditions of Theorem~\ref{thm35}. Thus Theorem ~\ref{thm32'} follows from  Theorem~\ref{thm35} and Corollary~\ref{cor40'}.
\end{proof}

\begin{remark}There exists a complete analog of Galois theory for linear homogeneous differential equations (see \cite{[10]}). Theorem~\ref{thm32} can be proved using this theory. The differential Galois group of a nonlinear homogeneous differential equation (\ref{eq10'}) could be very small and for such equation a complete analog of Galois theory  does not exist. Thus Theorem~\ref{thm32'} can not be proved in a similar way.
\end{remark}

\subsubsection{ Finite extensions of fields of rational functions}\label{sec3.3}

Here we  discuss finite extensions of the field $K(y)$ of rational functions over a subfield $K$ of the field of meromorphic function on a connected Riemann surface $U$. { We also state  Theorem~\ref{thm42} (generalized Newton's Theorem) which plays  a crucial role for this chapter.}

Let $F$ be  an extension of $K(y)$ by a root $z$ of a degree $m$ polynomial $P(z)\in (K[y])[z]$ over the ring $K[y]$ irreducible over the field $K(y)$. Let $X$ be the product $U\times \Bbb C^1$ where $\Bbb C^1$ is the standard complex line with the coordinate function $y$. An element of the field $K(y)$ can be considered as a meromorphic function on $X$. One can associate with the element $z\in F$  a multivalued  algebroid  function on $X$ defined by th equation $P(z)=0$. Let
$D(y)$ be the discriminant of the polynomial $P$. Let $\Sigma\subset U\times \Bbb C^1=X$ be the hypersurface defined by equation $p_m (y)\cdot D(y)=0$ where $p_m(y)$ is the leading coefficient of the polynomial $P$.

\begin{lem}\label{lem41} 1) About a point $x\in X\setminus \Sigma$ the equation $P(z)=0$ defines $m$  germs $z_i$ of analytic functions whose values at  $x$ are simple roots of polynomial $P$.

2) Let $x$ be the point $(a,y)\in U\times \Bbb C^1\setminus \Sigma$. Then the field $F$ is isomorphic to the extension $K_a(z_i)$ of the field $K_a$ of  germs at $a\in U$ of  functions from the field $K$ (considered as germs at $x=(a,y)$ of functions independent of $y$) extended by the germ $z_i$ at $x$ satisfying the equation $P(z)=0$.
\end{lem}
\begin{proof} The statement 1) follows from the Implicit Function Theorem. The statement 2) follows from 1).
\end{proof}

Below we state Theorem~\ref{thm42} which is a generalization of Newton's Theorem  about the expansion of an algebraic functions as convergent Puiseux series. It is stated  without proof  (see comments in section~\ref{secondintro}).

We use notations introduced in the beginning of this section. Let $z$ be  an element  satisfying a polynomial equation $P(z)=0$ over the ring $K[y]$,  where $K$ is a subfield  of the field of meromorphic functions on $U$. Then there exists {\it a finite extension $K_P$ of the field $K$ associated with the element $z$} such that the following theorem holds.

\begin{thm}\label{thm42} There is a finite covering $\pi :U_P\rightarrow U\setminus O_P$ where $O_P\subset U$ is a discrete subset, such that  the following properties hold:

1) the extension $K_P$ can be realized by a subfield of the field of meromorphic functions on $U_P$ containing the field $\pi^*K$ isomorphic to $K$.

2) there is a continuous positive function $r:U\setminus O_P\rightarrow \Bbb R$ such that in the  open domain $W\subset (U\setminus O_P)\times \Bbb C^1$ defined by the inequality $|y|>r(a)$  all $m$ germs of  $z_i$ at a point $(a,y)$ can be developed into converging   Puiseux series

\begin{eqnarray}
z_i=z_{i_k}  y^{\frac {k}{p}}+z_{i_{k-1}}y^{\frac {k-1}{p}}+\dots \label{eq14}
\end{eqnarray}
whose coefficients $z_{i_j}$ are germs of analytic functions at the point $a\in U_P$ having analytical  continuation as regular functions on  $U_P$  belonging to the field  $K_P$.
\end{thm}

For the psecial case when $K$ is a subfield of the field of complex numbers, it is natural to refer to  Theorem~\ref{thm42} as Newton's Theorem. One can consider
$K$ as a field of constant functions on any connected Riemann surface $U$.  One can chose $O_P$ to be the empty set, $U_P$ to be equal $U$,  the projection  $\pi:V_P\rightarrow U$ to be  the identity map, the function $r:U\rightarrow \Bbb R$  to be a big enough constant. In this case Theorem~\ref{thm42} states that an algebraic function $z$ has a Puiseux expansion at infinity whose coefficients belong to a finite extension $K_P$ of the field $K$. This statement can be proved by Newton's polygon method.

Let $F$ be an extension of $K(y)$ by a root $z$ of the polynomial $P$ and let $K_P$ be the finite extension of the field $K$ introduced in Theorem~\ref{thm42}. The extension $F_P$ of the field $K_P(y)$ by $z$ is easy to deal with. Denote the product $U_P\times \Bbb C^1$ by $X_P$.

\begin{lem}\label{lem43}  Let $x\in X_P$ be the point $(a,y_0)\in U_P\times \Bbb C^1$. Then the field $F_P$ is isomorphic to the extension $K_{P,a}(z_i)$ of the field $K_{P,a}$ of  germs at $a\in U_P$ of  functions from the field $K_P$ (considered as germs at $x=(a,y_0)$ of functions independent of $y$) extended by the germ  at $x$  of the function $z_i$ defined by (\ref{eq14}).
\end{lem}

Lemma~\ref{lem43} follows from Theorem~\ref{thm42}.

\subsubsection{Extension by one transcendental element.}

Let $U$ be a connected Riemann surface and let $K$ be a differential field of meromorphic functions on $U$. Let $\Bbb C^1$ be the standard complex  line with the coordinate function $y$. Elements of the field $K(y)$ of rational functions over $K$ could be considered as meromorphic  functions on $X=U\times \Bbb C^1$.

In the field $K(y)$   there are two natural operations of differentiation. The first operation $R(y)\rightarrow \frac{\partial R}{\partial x}(y)$ is defined as follows: the derivative $\frac{\partial }{\partial x}$ of the independent variable $y$ is equal to zero, and derivative $\frac{\partial }{\partial x}$ of an element $a\in K$ is  equal to its derivative $a'$ in the field $K$. For the second operation $R(y)\rightarrow \frac{\partial R}{\partial y}(y)$ the derivative of an element $a\in K$  is equal to zero and the derivative of the independent variable $y$ is equal to one.

Let $K\subset F$  be  differential fields and let $\theta\in F$  be a transcendental  element over $K$. Assume that  $\theta '\in K\langle \theta\rangle$. Under this assumption the field $K\langle\theta\rangle$ has a following description.

\begin{lem}\label{lem44} 1) The map  $\tau:K\langle \theta \rangle \rightarrow K(y)$ such that $\tau (\theta)=y$ and $\tau(a)=a$ for $a\in K$ provides an isomorphism between the field  $K\langle \theta \rangle$ considered  without the operation of differentiation and the field $K(y)$ of rational functions over $K$.

2) If $\tau (\theta')=w\in K(y)$ then for any $R\in K(y)$ and $z\in K\langle \theta \rangle $ such that $\tau(z)=R$ the following identity holds
\begin{eqnarray}
\tau (z')=\frac{\partial R}{\partial x}+ \frac{\partial R}{\partial y}w. \label{eq15}
\end{eqnarray}
\end{lem}

\begin{proof}The first claim of the lemma is straightforward. The second claim follows from the chain rule.
\end{proof}

Let $\Theta\subset X=U\times \Bbb C^1$  be the graph of function $\theta:U\rightarrow \Bbb C^1$. The following lemma is straightforward.

\begin{lem}\label{lem45} The differential field $K\langle\theta \rangle$ is isomorphic to the field  $K(y)|_\Theta$ obtained by restriction to $\Theta$ of functions from the field $K(y)$ equipped with the differentiation given by (\ref{eq15}). For any point $a\in \Theta$ The differential field $K\langle\theta \rangle$ is isomorphic to the differential field of germs at $a\in \Theta$ of functions from $K(y)|_\Theta$.
\end{lem}

\subsubsection{An extension by integral}

Here we consider  extensions of transcendence degree one of a differential field $K$ containing an integral $y$ over $K$ which does not  belong to $K$, $y\notin K$.
\paragraph{A pure transcendental  extension by integral}\label{sec3.5.1}
Let $\theta$ be an integral over $K$,i.e $\theta'=f\in K$. Assume that $\theta$ is a transcendental element over $K$. \footnote{It is easy to check that if an integral $\theta$ over $K$ does not belong to $K$, then $\theta$ is a transcendental element over $K$ (see \cite{[6]}). We will not use this fact.}

\begin{lem}\label{lem46} 1) The field $K\langle \theta \rangle$ is isomorphic to the field $K(y)$ of rational functions over $K$ equipped with the following differentiation
\begin{eqnarray}
R'=\frac{\partial R}{\partial x}+ \frac{\partial R}{\partial y}f.\label{eq16}
\end{eqnarray}
2) For every complex number $\rho\in \Bbb C$ the map $\theta\rightarrow \theta+\rho$ can be extended to the unique isomorphism $G_\rho:K\langle \theta \rangle\rightarrow K\langle \theta \rangle$ which fixes elements of the field $K$.

3) Each isomorphism of $K\langle \theta \rangle$ over $K$ is an isomorphism $G_\rho$ for some $\rho\in \Bbb C$. Thus the Galois group of $K\langle \theta \rangle$ over $K$ is the additive group of complex numbers $\Bbb C$.
\end{lem}

\begin{proof} The claim 1) follows from Lemma 44. For any $\rho \in \Bbb C$ the element $\theta_\rho=\theta+\rho$ is a transcendental element over $K$ and $\theta_\rho'$ equals to $f$.  Thus the claims 2) is correct.  The claim 3) follows from 2) because if $y'=f$ then $y=\theta_\rho$ for some $\rho\in \Bbb C$.
\end{proof}

\paragraph{A generalized extension by an integral}\label{sec3.5.2}

According to Lemma~\ref{lem46} the differential field $K\langle\theta\rangle$ is isomorphic to the field $K(y)$ with the differentiation given by (\ref{eq16}). Let $F$ be an extension of $K\langle\theta\rangle$ by an  element $z\in F$   which satisfies some equation $\tilde P(z)=0$ where $\tilde P$ is an irreducible polynomial over $K\langle \theta \rangle$. The isomorphism between $K\langle\theta\rangle$ and $K(y)$ transforms the polynomial $\tilde P$ into some polynomial $P$ over $K(y)$. Below we  use notation from section~\ref{sec3.3} and deal we the multivalued algebroid function $z$ on $X$  defined by $P(z)=0$.

Assume that at a point $x\in X$ there are germs of analytic functions $z_i$ satisfying the equation $P(z_i)=0$. Let $\theta_\rho$ be the function $(\theta+\rho):U_P\rightarrow \Bbb C^1$ and let $\Theta_\rho \subset X=U\times\Bbb C^1$ be its graph. The point  $x=(p,q)\in U\times \Bbb C^1$ belongs to the graph $\Theta_{\rho(x)}$ for $\rho(x)=q-\theta(p)$.

Let $K(y)|_{\Theta_{\rho(x)}}$ be the differential field of germs at the point $x\in \Theta_{\rho(x)}$ of  restrictions on $\Theta_{\rho(x)}$ of functions from the field $K(y)$ equipped with the differentiation given by (\ref{eq16}).

\begin{lem}\label{lem47} The differential field $F$ is isomorphic to the finite extension of the differential field  $K(y)|_{\Theta_{\rho(x)}}$ obtained by adjoining the germ at $x\in \Theta_{\rho(x)}$ of the restriction to $\Theta_{\rho(x)}$ of an analytic  germ $z_i$ satisfying~$P(z_i)=0$.
\end{lem}
\begin{proof} For the trivial extension  $F=K\langle \theta \rangle$  Lemma~\ref{lem47} follows from  Lemmas~\ref{lem45} and~\ref{lem46}. Theorem~\ref{thm19} allows one  to complete the proof for non trivial finite extensions $F$ of $K\langle \theta \rangle$.
\end{proof}

According to  section~\ref{sec3.3}, with the polynomial $P$ over $K(y)$ one can associate the finite extension $K_P$ of the field $K$ and the Riemann surface $U_P$ such that Theorem~\ref{thm42} holds. Since $K$ is  functional differential field the field $K_P$ has a natural structure of  functional differential field. Below we will apply Lemma~\ref{lem47} taking instead of $K$ the field $K_P$ and considering the extension $F_P\supset K_P\langle \theta \rangle$ by the same algebraic element $z\in F$. The use of $K_P$ instead of $K$ allows one to apply the expansion (14) for~$z_i$.

\begin{thm}\label{thm48} Let $x\in X_P= U_P\times \Bbb C^1$ be a point $(a,y_0)$ with $|y|>>0$. The differential field $F_P$ is isomorphic to the extension of the differential field  of  germs at the point $a\in U_P$ of functions from the differential field $K_P$ by the  following germs: by the germ at $a$ of the integral $\theta_{\rho(x)}$ of the function $f\in K$, and by a germ at $a$ of the composition $z_i(\theta_\rho)$ where $z_i$ is a germ at $x$ of a function given by a Puiseux series (14).
\end{thm}

\begin{proof}Theorem~\ref{thm48} follows from Lemma~\ref{lem47} and Theorem~\ref{thm42}.
\end{proof}

\paragraph{Solutions of equations in a generalized extension by integral}\label{sec3.5.3}

Here we discuss lemma~\ref{lem49} providing an important step for our proof of Theorem~\ref{thm35}. We will use notations from sections~\ref{sec3.5.1} and \ref{sec3.5.2}.

Let $T(u,u',\dots,u^{(n)}$ be a polynomial in independent function $u$ and its derivatives with coefficients from the functional differential field $K$. Consider the equation
\begin{eqnarray}
T(u,u',\dots,u^{(N)})=0. \label{eq17}
\end{eqnarray}
In general  the derivative of the highest order $u^{(N)}$ cannot  be represented as a function  of other derivatives via the relation (\ref{eq17}). Thus even existence of local solutions of (\ref{eq17}) is problematic and we have no information about global behavior of its solutions.

Assume that (\ref{eq17}) has a solution $z$ in a generalized extension by integral $F\supset K\langle \theta \rangle$ of $K$. The solution $z$ has a nice global property: it is a meromorphic function on a Riemann surface $V$ with a projection $\pi:V\rightarrow U$ which proves a locally trivial covering above $U\setminus O$, where $O\subset U$ is discrete subset.

Moreover the existence of a solution $z$ implies the existence of a  family $z(\rho)$ of similar solutions  depending on a parameter $\rho$: one  obtains such family of solutions  by using an integral $\theta+\rho$ instead of the integral $\theta$ (see Lemma~\ref{lem47}). If the parameter $\rho$ has big absolute value $|\rho|>>0$ for a point $a\in U_P$ of the germ $z(\rho)$ can be expanded in the Piuseux series in $\theta_\rho$:

\begin{eqnarray}
z_i(\rho)=z_{i_k}  \theta_\rho^{\frac {k}{p}}+z_{i_{k-1}}\theta_\rho^{\frac {k-1}{p}}+\dots \label{eq18}
\end{eqnarray}

The series is converging and so it can be differentiate using the relation $\theta_\rho'=f$.

\begin{lem}\label{lem49} If $z_{i_k}'\neq 0$ then the leading term of the Puiseux series for $z_i(\rho)'$ is $z_{i_k}'  \theta_\rho^{\frac {k}{p}}$. Otherwise the leading term has degree $<\frac {k}{p}$. The leading term of the derivative of any order of $z_{i_k}$ has degree $\leq \frac {k}{p}$.
\end{lem}

Let us plug into the differential polynomial $T(u,u',\dots,u^{(N)}) $ the germ (\ref{eq18}) and develop the result into Puiseux series in $\theta_\rho$.  If the germ  $z_i(\rho)$ is a solution of the equation (\ref{eq17}) then all terms of this Puiseux series are equal to zero. In particular the leading coefficient is zero. This observation is an important step for  proving  Theorem~\ref{thm35}.

\subsubsection{An extension by exponential of integral}

Here we consider  extensions of transcendental degree one of a differential field $K$ containing an exponential integral $y$ over $K$ which is not algebraic over  $K$.

\paragraph{3.6.1. A pure transcendental  extension by an exponential integral}\label{sec3.6.1}

Let $\theta$ be an an exponential of integral over $K$, i.e $\theta'=f\theta$ where $f \in K$. Assume that $\theta$ is a transcendental element over $K$. \footnote{It is easy to check that if  an exponential of integral $\theta$ over $K$ is algebraic over  $K$, then $\theta$ is a radical over $K$, i.e. $\theta^k\in K$ for some positive integral $k$ (see \cite{[6]}). We will not use this fact.}

\begin{lem}\label{lem50} 1) The field $K\langle \theta \rangle$ is isomorphic to the field $K(y)$ of rational functions over $K$ equipped with the following differentiation
\begin{eqnarray}
R'=\frac{\partial R}{\partial x}+ \frac{\partial R}{\partial y}fy. \label{eq19}
\end{eqnarray}

2) For every complex number $\mu\in \Bbb C^*$ not equal to zero  the map $\theta\rightarrow \mu \theta$ can be extended to the unique isomorphism $G_\mu:K\langle \theta \rangle\rightarrow K\langle \theta \rangle$ which fixes elements of the field $K$.

3) Each isomorphism of $K\langle \theta \rangle$ over $K$ is an isomorphism $G_\mu$ for some $\mu\in \Bbb C^*$. Thus the Galois group of $K\langle \theta \rangle$ over $K$ is the multiplicative group of complex numbers $\Bbb C^*$.
\end{lem}

\begin{proof} Statement 1) follows from Lemma~\ref{lem44}. For any $\mu \in \Bbb C^*$ the element $\theta_\mu=\mu \theta$ is a transcendental element over $K$ and $\theta_\mu'=f\theta$.  Thus   2) is correct.  Statement  3) follows from 2) because if $y'=fy$ and $y\neq 0$ then $y=\theta_\mu$ for some $\mu\in \Bbb C^*$.
\end{proof}

\paragraph{A generalized extension by exponential of integral}

According to Lemma~\ref{lem50} the differential field $K\langle\theta\rangle$ is isomorphic to the field $K(y)$ with the differentiation given by (\ref{eq19}). Let $F$ be an extension of $K\langle\theta\rangle$ by an  element $z\in F$  which satisfies some equation $\tilde P(z)=0$ where $\tilde P$ is an irreducible polynomial over $K\langle \theta \rangle$. The isomorphism between $K\langle\theta\rangle$ and $K(y)$ transforms the polynomial $\tilde P$ into some polynomial $P$ over $K(y)$. Below we  use notation from  section~\ref{sec3.3} and we deal with  the multivalued algebroid function $z$ on $X$  defined by $P(z)=0$.

Assume that at a point $x\in X$ there are germs of analytic function $z_i$ satisfying the equation $P(z_i)=0$. Let $\theta_\mu$ be the function $(\mu \theta):U_P\rightarrow \Bbb C^1$ and let $\Theta_\mu \subset X=U\times\Bbb C^1$ be its graph. The point  $x=(p,q)\in U\times \Bbb C^1$ where $q\neq 0$ belongs to the graph $\Theta_{\mu(x)}$ if $\mu(x)=q\cdot\theta(p)^{-1}$.

Let $K(y)|_{\Theta_{\mu(x)}}$ be the differential field of germs at the point $x\in \Theta_{\mu(x)}$ of  restrictions on $\Theta_{\mu(x)}$ of functions from the field $K(y)$ equipped with the differentiation given by (\ref{eq19}).

\begin{lem}\label{lem51}The differential field $F$ is isomorphic to the finite extension of the differential field  $K(y)|_{\Theta_{\mu(x)}}$ obtained by adjoining the germ at $x\in \Theta_{\mu(x)}$ of the restriction to $\Theta_{\mu(x)}$ of an analytic  germ $z_i$ satisfying $P(z_i)=0$.
\end{lem}
    \begin{proof} For the trivial extension  $F=K\langle \theta \rangle$  Lemma~\ref{lem51} follows from  Lemmas~\ref{lem45},\ref{lem50}. Theorem~\ref{thm19} allows one to complete the proof for non trivial finite extensions $F$ of $K\langle \theta \rangle$.
\end{proof}

According to section~\ref{sec3.3} with the polynomial $P$ over $K(y)$ one can associate the finite extension $K_P$ of the field $K$ and the Riemann surface $U_P$ such that Theorem~\ref{thm42} holds. Since $K$ is a functional differential field the field $K_P$ has a natural structure of  functional differential field. Below we will apply Lemma~\ref{lem51} taking instead of $K$ the field $K_P$ and considering the extension $F_P\supset K_P\langle \theta \rangle$ by the same algebraic element $z\in F$. The use of $K_P$ instead of $K$ allows to apply the expansion (\ref{eq14}) for~$z_i$.

\begin{thm}\label{thm52} Let $x\in X_P= U_P\times \Bbb C^1$ be a point $(a,y_0)$ with $|y|>>0$. The differential field $F_P$ is isomorphic to the extension of the differential field  of  germs at the point $a\in U_P$ by an exponential of an integral $\theta_{\mu(x)}$, where $\theta_{\mu(x)}'=f \theta_{\mu(x)}$ for the function $f\in K$, and by a germ at $a$ of the composition $z_i(\theta_\mu)$ where $z_i$ is a germ at $x$ of a function given by a Puiseux series (\ref{eq14}).
\end{thm}
\begin{proof}Theorem~\ref{thm52} follows from Lemma
~\ref{lem51} and Theorem~\ref{thm42}.
\end{proof}

\paragraph{Solutions of equations in a generalized extension by exponential of integral}

 Here we discuss Lemma~\ref{lem53} providing an important step for our proof of Theorem~\ref{thm35}. Assume that (\ref{eq17}) has a solution $z$ in a generalized extension by exponential of integral $F\supset K\langle \theta \rangle$ of $K$. The solution $z$ has a nice global property: it is a meromorphic function on a Riemann surface $V$ with a projection $\pi:V\rightarrow U$ which proves a locally trivial covering above $U\setminus O$, where $O\subset U$ is a discrete subset.

Moreover the existence of a solution $z$ implies the existence of  a family $z(\mu)$ of similar solutions  depending on a parameter $\mu\in \Bbb C^*$: one  obtains such family of solutions  by using an exponential of integral $\mu \theta$ instead of the exponential of integral $\theta$ (see \ref{lem51}). If the parameter $\mu$ has big absolute value $\mu|>>0$ for a point $a\in U_P$ of the germ $z(\mu)$ can be expand in the Piuseux series in $\theta_\mu$:

\begin{eqnarray}
z_i(\mu)=z_{i_k}  \theta_\mu^{\frac {k}{p}}+z_{i_{k-1}}\theta_\mu^{\frac {k-1}{p}}+\dots \label{eq20}
\end{eqnarray}

The series is converging and so it can be differentiate using the relation $\theta_\mu'=f\theta_\mu$.

\begin{lem}\label{lem53} If $z_{i_k}'+\frac{k}{p}z_{i_k}\neq 0$ then the leading term of the Piueux series for $z_i(\mu)'$ is $(z_{i_k}'+\frac{k}{p}z_{i_k})  \theta_\mu^{\frac {k}{p}}$. Otherwise the leading term has degree $<\frac {k}{p}$. The leading term of derivative of any order of $z_{i_k}$ has degree $\leq \frac {k}{p}$.
\end{lem}

Let us plug into the differential polynomial $T(u,u',\dots,u^{(N)}) $ the germ  (\ref{eq20}) and develop the result as a Puiseux series in $\theta_\mu$.  If the germ  $z_i(\mu)$ is a solution of the equation (\ref{eq17}) then all terms of this Piuiseux series are equal to zero. In particular the leading coefficient is zero. This observation is an important step for  proofing Theorem~\ref{thm35}.

\subsubsection{Proof of Rosenlicht's theorem}

Here we  complete an elementary  proof of Theorem~\ref{thm35} discovered by Maxwell Rosenlicht \cite{[13]}.
We will proof first the simpler Theorems~\ref{thm54} and~\ref{thm55} of a similar   nature.

\begin{thm}\label{thm54}
Assume that the equation (\ref{eq11}) over a functional differential field $K$ has a solution $z\in F$ where $F$ is a generalized extension by integral  of  $K$. Then (\ref{eq11}) has a solution in the algebraic extension $K_P$  of $K$ associated with the element $z\in F$.
\end{thm}

\begin{proof} If the constant term of the differential polynomial $T(u,u',u'',\dots)= u^n-Q(u, u', u'', \dots)$ is equal to zero, then (\ref{eq11}) has solution $u\equiv 0$ belonging to $K$. In this case we have nothing to prove.

Below we will assume that the constant term $T_0$ of $T$ is not equal to zero.
Thus the differential polynomial $T$ has two special terms: the term $u^n$ which is the only term of highest degree $n$ and the term $T_0$ which is the only term of smallest degree zero.

Assume that (\ref{eq11}) has a solution $z$ in a generalized extension by integral $F\supset K\langle \theta \rangle$ of $K$. According to section~\ref{sec3.5.3} the existence of such a solution $z$  implies the existence of  family $z(\rho)$ of germs of solutions  depending on a parameter $\rho$ such that when the    absolute value $\rho|$ is big enough $z(\rho)$ can be expanded in  Puiseux series (\ref{eq18})  in $\theta_\rho$.

We will show that the degree $\frac{k}{p}$ of the leading term in (\ref{eq18}) is equal to zero and the leading coefficient $z_{i_0}\in K_P$ satisfies (\ref{eq11}). This will prove Theorem~\ref{thm54}.

According to Lemma~\ref{lem49} the  leading term of the derivative of any order of $z_{i_k}$ has degree $\leq \frac {k}{p}$. Thus the leading term of the  Puiseux series obtained by plugging (\ref{eq18}) instead of $u$ into differential polynomial $Q$ has degree $< n \frac {k}{p}$. The leading term of the Puiseux series obtained by arising (\ref{eq18} ) to the  $n$-th power is equal to $n \frac {k}{p}$. If $\frac {k}{p}>0$ this term can not be canceled after plugging (\ref{eq18}) instead of $u$ into differential polynomial $T$. Thus the degree $\frac {k}{p}$ cannot be positive.

Let us plug (\ref{eq18}) into the differential polynomial $T(u,u',\dots)-T_0 $. We will obtain a Puiseux series of negative degree if $\frac {k}{p}<0$. Thus the term $T_0$ in the sum $(T-T_0)+T_0$ can not be canceled. Thus $\frac {k}{p}$ cannot be negative.

We proved that $\frac {k}{p}=0$. If in this case we plug (\ref{eq18}) into the differential polynomial $T(u,u',\dots)$ we obtain a Puiseux series having only one term of nonnegative degree which is equal to zero.
From Lemma~\ref{lem49} it is easy to see that this term equals to $T(z_{i_0},z_{i_0}',\dots) \theta_\rho^0$. Thus $z_{i_0}\in K_P$ is a solution of (\ref{eq11}). Theorem~\ref{thm54} is proved. \end{proof}

\begin{thm}\label{thm55}
Assume that  equation (\ref{eq11}) over a functional differential field $K$ has a solution $z\in F$ where $F$ is a generalized extension by an exponential of integral  of  $K$. Then (\ref{eq11}) has a solution in the algebraic extension $K_P$  of $K$ associated with the element $z\in F$.
\end{thm}
\begin{proof}Theorem~\ref{thm55} can be proved exactly in the same way as Theorem ~\ref{thm54}. Instead of Lemma~\ref{lem49} one has to use  Lemma~\ref{lem51}. In the case when  leading term of the Puiseux expansion of $z_\mu$ has degree zero, the leading coefficient of its derivative equals to $z_{i_0}'$ (see Lemma~\ref{lem51}). That is why the case $\frac {k}{p}=0$ in Theorem~\ref{thm55} can be treated exactly in the same way as in Theorem\ref{thm54}.
\end{proof}
Now we ready to prove  Theorem~\ref{thm35}.

\begin{proof} {Proof of Theorem 35} By assumption the equation (\ref{eq11})  has a solution $z\in F$ where $F$ is an extension of $K$ by generalized quadratures. By definition there is a chain  $K=F_0\subset \dots\subseteq F_m$ such that $F\subset F_m$ and for every $i=0$, $\dots$, $m-1$  or $F_{i+1}$ is a finite extension of $F_i$, or $F_{i+1}$ is  a generalized extension by integral of $F_i$, or $F_{i+1}$ is  a generalized extension by exponential integral of $F_i$. We prove Theorem~\ref{thm35} by induction in the length $m$ of the chain of extension. For $m=1$ Theorem ~\ref{thm35} follows from Theorem~\ref{thm54}, or from Theorem~\ref{thm55}. Assume that  Theorem~\ref{thm35} is true for $m=k$. A chain $F_0 \subset F_1\subset \dots\subset F_{k+1}$ provides the chain $F_1\subset \dots\subset F_{k+1}$ of extensions of of length $k$ for the field $F_1$. Thus (\ref{eq11}) has an algebraic solution $z$ over the field $F_1$.
The extension $F_0\subset \tilde F_1$, where $\tilde F_1$ is the extension of $F_1$ by the element $z$, is either an algebraic extension, or extension by generalized integral or extension by generalized exponential of integral. Thus for the extension $F_0\subset \tilde F_1$ Theorem ~\ref{thm35} holds. We completed the inductive proof of Theorem~\ref{thm35}.
\end{proof}

\section{Topological Galois Theory}\label{chap2}

\subsection{Introduction}

In this section we present an outline   of topological Galois theory based on the  book~\cite{[6]}. The theory   studies topological obstructions to solvability  of equations ``in finite terms" i.e. to their solvabiltiy by radicals, by elementary functions, by quadratures and by functions belonging to other Liouvillian classes.

As was discovered by  Camille Jordan  the
Galois group of  an algebraic equation over the field of rational functions of several complex variables has a topological meaning: it is isomorphic to the monodromy group of the algebraic function defined by this algebraic equation.  Therefore  the monodromy group is
responsible for the representability of an algebraic function
by radicals.

In the section~\ref{ch2sec1} we present a  topological proof of the nonrepresentability of algebraic functions by radicals. This proof is based on my old paper~\cite{[4]}. It contains a germ of  topological Galois theory.

Not only algebraic functions have a
monodromy group.  It is defined for any solution of a linear differential equation whose  coefficients are rational  functions and for many more functions, for which the Galois group does not make sense.

It is thus natural to try using the monodromy group
for these functions instead of the Galois group to prove that
they do not belong to a certain Liouvillian class. This particular
approach is implemented in   topological
Galois theory (see~\cite{[6]}), which has a one-dimensional version and a multidimensional version.

In the sections~\ref{ch2sec2} and~\ref{ch2sec3} we present an outline of the one-dimensional version and an outline of the multidimensional version. These sections contain definitions, statements of results and comments to them. Basically no proofs are presented.

\subsection{On Representability of Algebraic Functions by Radicals}\label{ch2sec1}
\subsubsection{Introduction}\label{ch2sec1ss1}  This section is dedicated to a self contained simple proof of the classical criteria for  representability of  algebraic functions of several complex variables by  radicals. It also contains a  criteria for  representability of  algebroidal functions by composition of single-valued analytic functions and radicals, and a  result related to the 13-th Hilbert problem.

Consider an algebraic equation
\begin{eqnarray}
 P_n y^n+P_{n-1}y^{n-1}+\dots +P_0=0, \label{eq2.1}
 \end{eqnarray}
 whose coefficients  $P_n,\dots, P_0$ are polynomials  of $N$ complex variables $x_1,\dots, x_N$.
Camille Jordan discovered that  the  Galois group of the equation (\ref{eq2.1}) over the field $\mathcal R$ of rational functions of $x_1,\dots,x_N$ has a topological meaning (see theorem~\ref{thm2.3} below): it is isomorphic to the {\it monodromy group} of the equation (\ref{eq2.1}).

According to   the Galois theory,  equation (\ref{eq2.1}) is solvable
by  radicals over the field $R$ if and only if
its Galois group is solvable. If the equation (\ref{eq2.1}) is irreducible it defines a multivalued algebraic function $y(x)$.
The Galois theory and Theorem~\ref{thm2.3} imply  the  following criteria for representability of an algebraic function by radicals, which consists of  two statements:

1) {\it  If the  monodromy group of an algebraic function  $y(x)$ is
solvable, then $y(x)$ is representable by radicals.}

2) {\it \it  If the  monodromy group of an algebraic function  $y(x)$ is not
solvable, then $y(x)$ is not representable by radicals.}

 We reduce the first statement  to linear algebra (see Theorem~\ref{thm2.10} below) following the book~\cite{[6]}.

We prove the second statement  topologically without using Galois theory. Vla\-di\-mir Igorevich Arnold found the first topological proof of this statement~\cite{[1]}. We use another topological approach (see Theorem~\ref{thm2.15} below) based on the paper~\cite{[4]}.  This paper contains  the first result of  topological Galois theory~\cite{[6]} and it gave a hint for its further development.

\subsubsection{Monodromy group and Galois group}\label{ch2sec1ss2}

Consider the  equation (\ref{eq2.1}). Let $\Sigma \subset \Bbb C^N$ be the hypersurface defined by  equation $P_n J=0$, where $P_n$ is the leading coefficient and $J$ is the discriminant of the equation (\ref{eq2.1}). The  {\it  monodromy group } of the equation (\ref{eq2.1})  is the group of all  permutations of  its solutions  which are induced by
motions  around the singular set $\Sigma $ of the equation (\ref{eq2.1}). Below we discuss this definition more precisely.

At a point $x_0\in \Bbb C^N\setminus \Sigma $ the set $Y_{x_0}$   of all germs   of analytic  functions  satisfying  equation (\ref{eq2.1}) contains exactly $n$ elements, i.e. $Y_{x_0}=\{y_1,\dots,y_n\}$. Indeed, if  $P_n(x_0)\neq 0$ then for    $x=x_0$   equation (\ref{eq2.1}) has $n$ roots counted with multiplicities. If in addition   $J(x_0)\neq 0$ then all these roots are simple. By the implicit function theorem each simple root can be extended to a germ of a regular function satisfying the equation  (\ref{eq2.1}).

Consider a closed curve  $\gamma $ in
$\Bbb C^N\setminus \Sigma $ beginning and ending at the point $x_0$. Given a germ $y\in Y_{x_0}$ we can continue it along the loop $\gamma$ to obtain another germ $y_\gamma\in Y_{x_0}$.
Thus each such loop $\gamma$   corresponds to a permutation $S_{\gamma}:Y_{x_0}\rightarrow Y_{x_0}$
 of the set $Y_{x_0}$ that maps a germ $y\in Y_{x_0}$ to the germ $y_{\gamma}\in Y_{x_0}$. It is  easy to see that the map $\gamma\rightarrow S_\gamma$ defines a homomorphism   from the fundamental
group $\pi _1(\Bbb C^N\setminus \Sigma,x_0)$ of the domain $\Bbb C^N\setminus \Sigma $ with the base point $x_0$ to the group $S(Y_{x_0})$ of permutations. The {\it
monodromy group} of the equation (\ref{eq2.1}) is the image of
the fundamental group   in the group
$S(Y_{x_0})$ under this homomorphism.

\begin{remark} Instead of the point $x_0$ one can choose any other point $x_1\in \Bbb C^N\setminus \Sigma$. Such a choice will not change the    monodromy group  up to an isomorphism. To fix this isomorphism one can choose any curve $\gamma:I\rightarrow \Bbb C^N\setminus \Sigma$ where $I$ is the segment $ 0\leq t\leq 1$ and $\gamma(0)=x_0$, $\gamma(1)=x_1$ and identify each germ $y_{x_0}$ of solution of (\ref{eq2.1}) with its continuation $y_{x_1}$ along $\gamma$.

Instead of the hypersurface $\Sigma$ one can choose any bigger algebraic hypersurface $D$, $\Sigma\subset D\subset \Bbb C^N$. Such a choice will not change the    monodromy group: one can slightly  move a curve $\gamma \in \pi_1(\Bbb C^N\setminus \Sigma, x_0)$ without changing the map $S_\gamma$ in such a way that $\gamma$ will not intersect $D$.
\end{remark}

The field of rational functions of $x_1,\dots, x_N$ is isomorphic to the field $\mathcal R$ of germs of rational functions at the point $x_0\in \Bbb C^N\setminus \Sigma $.
Consider the field extension  $\mathcal R \langle y_1,\dots,y_n\rangle$ of $\mathcal R$ by the  germs  $ y_1,\dots y_n$ at $x_0$  satisfying the equation (\ref{eq2.1}).

\begin{lem}\label{lem2.1} Every permutation $S_\gamma$ from the monodromy group can be uniquely extended to an automorphism of the   field $\mathcal R\{ y_1,\dots
,y_n\}$ over the field $\mathcal R$.
\end{lem}

\begin{proof} Every element $f\in \mathcal R \langle y_1,\dots,y_n\rangle$ is a rational function of $x, y_1,\dots,y_n$. It can be continued meromorphically along  the curve
$\gamma\in \pi_1(\Bbb C^m\setminus \Sigma,x_0)$  together with  $y_1,\dots,y_n$.
This continuation gives the required
automorphism, because the continuation preserves the arithmetical
operations  and every rational function returns back
to its original  values (since it is a single-valued valued function). The automorphism is unique because the extension is generated by $y_1,\dots, y_n$.
\end{proof}

By definition the {\it  Galois  group} of the equation (\ref{eq2.1})  is the group of all automorphisms of  the   field $\mathcal R\{ y_1,\dots,y_n\}$ over the field $\mathcal R$. According to Lemma~\ref{lem2.1} the monodromy group  of the equation (\ref{eq2.1}) can be considered as a subgroup of its Galois group. Recall that by definition a multivalued function $y(x)$ is  {\it algebraic} if all its meramorphic germs satisfy the same algebraic equation over the field of rational functions.

\begin{thm}\label{thm2.2} A germ $f\in \mathcal R \langle y_1,\dots,y_n\rangle$ is fixed under the monodromy action if and only if  $f\in \mathcal R$.
\end{thm}
\begin{proof} A germ $f\in \mathcal R \langle y_1,\dots,y_n\rangle$ is fixed under the monodromy action  if and only if $f$   is a germ of a single  valued function. The field $\mathcal R \langle y_1,\dots,y_n\rangle$ contains only germs of algebraic functions.  Any single valued algebraic function is a rational function.
\end{proof}

According to  the Galois theory Theorem~\ref{thm2.2} can be formulated in the following way.

\begin{thm}\label{thm2.3} The monodromy group of  the equation (\ref{eq2.1}) is isomorphic to the Galois group of the equation (\ref{eq2.1}) over the field $\mathcal R$.
\end{thm}
Below we will not rely on Galois theory. Instead  we will use Theorem~\ref{thm2.3} directly.

\begin{lem}\label{lem2.4} The monodromy group acts on the set $Y_{x_0}$ transitively if and only if the equation (\ref{eq2.1}) is irreducible over the field of rational functions.
\end{lem}

\begin{proof} Assume that there is a proper subset $\{y_1, y_2,\dots y_k\}$   of $Y_{x_0}$  invariant under the monodromy action. Then the basic symmetric functions $r_1=y_1+\dots +y_k$, $r_2=\sum_{i<j}y_i y_j,$ $\dots$, $r_k=y_1\cdot\dots\cdot y_k$ belong to the field $\mathcal R$. Thus $y_1, y_2,\dots y_k$ are solutions of the degree $k<n$ equation $y^k-r_1y^{k-1}t+\dots (-1)^kr_k=0$. So  equation (\ref{eq2.1}) is reducible. On the other hand if the equation (\ref{eq2.1}) can be represented as a product of two equations over $\mathcal R$ then their roots belong to  two complementary  subsets of $Y_{x_0}$ which are  invariant under the monodromy action.
\end{proof}

\begin{cor}\label{cor2.5} An irreducible  equation (\ref{eq2.1})  defines a multivalued algebraic function $y(x)$ whose set of germs at $x_0\in \Bbb C^N\setminus \Sigma $ is the set $Y_{x_0}$ and whose monodromy group coincides with the monodromy group of the equation (\ref{eq2.1}).
\end{cor}

Theorem~\ref{thm2.3}, Corollary~\ref{cor2.5} and the Galois theory immediately imply the following result.

\begin{thm}\label{thm2.6} An algebraic function whose monodromy group is solvable can be represent by rational functions using the arithmetic operations and radicals.rr
\end{thm}
A stronger version of Theorem~\ref{thm2.6}  can be proven  using linear algebra (see Theorem 10 in the next section).

\subsubsection{Action of  solvable groups and representability
by radicals}
In this section, we prove that if a finite solvable group $G$ acts on a $\Bbb C$-algebra $V$ by
automorphisms, then  all elements of $V$ can be expressed by the elements of the invariant subalgebra $V_0$ of $G$ by taking radicals and adding.

 This construction of a representation by radicals is based on linear algebra. More precisely we use the following well known statement:  any finite abelian
group of linear transformations of a finite-dimensional vector space over  $\Bbb C$
can be diagonalized in a suitable basis.

 \begin{lem}\label{lem2.7} Let $G$ be a finite abelian group of order $n$ acting by automorphisms
on  $\Bbb C$-algebra $V$.
Then every element of the algebra $V$ is representable as a sum of  elements
$x_i \in V $,  such that $x^n_i$
lies in the invariant subalgebra $V_0$ of $G$, i.e., in
the fixed-point set of the group $G$.
\end{lem}

\begin{proof} Consider a finite-dimensional vector subspace $L$ in the algebra $V$ spanned by
the $G$-orbit of an element $x$. The space $L$ splits into a direct sum $L = L_1+\dots+L_k$
of eigenspaces for all operators from $G$. Therefore, the vector $x$ can
be represented in the form $x = x_1+\dots + x_k$, where $x_1,\dots, x_k$ are eigenvectors
for all the operators from the group. The corresponding eigenvalues are $n$-th roots of
unity. Therefore, the elements $x^n_1,\dots, x^n_k$ belong to the invariant subalgebra $V_0$.
\end{proof}

\begin{defn} We say that an element $x$ of the  algebra $V$ is an $n$-th root of an
element~$a$ if $x^n=a$.
\end{defn}

We can now restate Lemma~\ref{lem2.7} as follows: every element $x $ of the algebra $V$
is representable as a sum of $n$-th roots of some elements of the invariant subalgebra.

\begin{thm}\label{thm2.8} Let $G$ be a finite solvable group of order $n$ acting by automorphisms
on  $\Bbb C$-algebra $V$. Then
every element $x$ of the algebra $V$ can be obtained from the elements of the invariant
subalgebra $V_0$ by takings $n$-th roots  and summing.
\end{thm}

We first prove the following simple statement about an action of a group on a set.
Suppose that a group $G$ acts on a set $X$, let $H$ be a normal subgroup of $G$, and denote by
$X_0$  the subset of $ X$ consisting of all points fixed under the action~of~$G$.

\begin{lem}\label{lem2.9}The subset $X_H$ of the set $X$ consisting of the fixed points under the action of the normal subgroup  $H$ is invariant under the action of $G$. There is a
natural action of the quotient group $G/H$ on the set $X_H$ with the fixed-point set~$X_0$.
\end{lem}

\begin{proof} Suppose that $g\in G$, $h \in H$. Then the element $g^{-1}hg$ belongs to the normal subgroup $H.$ Let $x\in X_H$. Then $g^{-1}hg (x)= x$, or $h(g(x))= g(x)$, which means
that the element $g(x)\in X $ is fixed under the action of the normal subgroup $H$. Thus
the set $X_H$ is invariant under the action of the group $G$. Under the action of $G$ on
$X_H$, all elements of $H$ correspond to the identity transformation. Hence the action
of $G$ on $X_H$ reduces to an action of the quotient group $G/H$.
\end{proof}

We now proceed with the proof of Theorem~\ref{thm2.8}.
\begin{proof} {\it (of Theorem~\ref{thm2.8})} Since the group $G$ is solvable, it has a chain of nested
subgroups $G = G_0\supset \dots\supset G_m = e$ in which the group $G_m$ consists of the
identity element $e$ only, and every group $G_i$ is a normal subgroup of the group
$G_{i-1}$. Moreover, the quotient group $G_{i-1}/G_i$ is abelian.
Let $V_0 \subset \dots\subset V_m =V$ denote the chain of invariant subalgebras of the
algebra $V$ with respect to the action of the groups $G_0,\dots,G_m$. By Lemma 9
the abelian group $G_{i-1}/G_i $ acts naturally on the invariant subalgebra $V_i $, leaving
the subalgebra $V_{i-1}$  pointwise fixed. The order $m_i$ of the quotient group $G_{i-1}/G_i$
divides the order of the group $G$. Therefore, Lemma~\ref{lem2.7} is applicable to this
action. We conclude that every element of the algebra $V_i$ can be expressed with the
help of summation and $n$-th root extraction by the elements of the algebra $V_{i-1}$.
Repeating the same argument, we will be able to express every element of the
algebra $V$ by the elements of the algebra $V_0 $ using a chain of summations and $n$-th
root extractions.
\end{proof}

\begin{thm}\label{thm2.10} An algebraic function whose  monodromy  is solvable can be represented by rational functions by root extractions and summations.
\end{thm}

\begin{proof}One can prove Theorem~\ref{thm2.10} by applying Theorem~\ref{thm2.8} to the monodromy action by automorphisms on the extension $\mathcal R \langle y_1,\dots,y_n\rangle$ with the field of invariants $\mathcal R$.
\end{proof}

\subsubsection{Topological obstruction to  representation by radicals}\label{secsecss2.4}
Let us introduce some notation.

By $G^m$ we denote the $m$-th commutator subgroup of the group $G$. For any  $m\geq $ the group $G^m$ is a normal subgroup in $G$.

By $F(D,x_0)$ we denote the fundamental group  of the domain $U=\Bbb C^N\setminus D$ with the base point $x_0\in U$, where $D$ is an algebraic hypersurface in $\Bbb C^N$.

Let $H(D,m)$ be the covering space  of the domain $\Bbb C^N\setminus D$ corresponding to the subgroup $F^m(D, x_0)$ of the fundamental group $F(D,x_0)$.

We will say that an algebraic function  is an {\it  $R$-function} if it  becomes a single-valued function on some covering $H(D,m)$.

\begin{lem}\label{lem2.11} If $m_1\geq m_2$ and $D_1\supset D_2$ then there is a natural projection
$\rho:H(D_1,m_1)\rightarrow H(D_2,m_2)$. Thus if a function $y$  becomes a single-valued function on $H(D_2,m_2)$ then it certainly becomes a single-valued function on $H(D_1,m_1)$.
\end{lem}

\begin{proof}Let $p_*: F(D_1,x_0)\rightarrow F(D_2,x_0)$ be the homomorphism induced by the embedding $p: \Bbb C^N\setminus D_1\rightarrow \Bbb C^N\setminus D_2$. Lemma~~\ref{lem2.11} follows from  the following  obvious  inclusions: $p_*^{-1}[F^{m_2}(D_2,x_0)]\subset F^{m_2}(D_1,x_0)$ and $F^{m_2}(D_1,x_0)\subset F^{m_1}(D_1,x_0)$.
\end{proof}

\begin{lem}\label{lem2.12} If $y_1$ and $y_2$ are $R$-function then $y_1+y_2$, $y_1-y_2$, $y_1\cdot y_2$ and  $y_1/y_2$  also are $R$-functions.
\end{lem}

\begin{proof}Assume that $R$-functions $y_1$ and $y_2$ become single-valued functions on the coverings $H(D_1,m_1)$ and  $H(D_2,m_2)$. By Lemma~\ref{lem2.11} the  functions $y_1$,$y_2$ become single-valued on the covering $H(D,m)$ where $D=D_1 \bigcup D_2$ and $m=\max (m_1, m_2)$. Thus the functions $y_1+y_2$, $y_1-y_2$, $y_1\cdot y_2$ and  $y_1/y_2$  also become single-valued on on the covering $H(D,m)$.
The proof is completed since  $y_1+y_2$, $y_1-y_2$, $y_1\cdot y_2$ and  $y_1/y_2$ are algebraic functions.
\end{proof}

\begin{lem}\label{lem2.13} Composition of an $R$-function with the degree $q$ radical  is an $R$-function.
\end{lem}

\begin{proof}Assume that the function $y$ defined by (\ref{eq2.1}) is  $R$-function which becomes a single-valued function on the covering $H(D_1,m)$.  Let $D_2\subset \Bbb C^N$ be the hypersurface, defined by the equation $P_n P_0=0$, where $P_n$ and $P_0$ are the leading coefficient and the constant term of the equation (\ref{eq2.1}). According to Lemma~\ref{lem2.11} the function $y$ becomes a single-valued function on the covering $H(D,m)$ where $D=D_1\bigcup D_2$. Let $h_0\in H(D,m)$  be a point whose image under the natural projection $\rho:H(D,m)\rightarrow \Bbb C^N\setminus D$ is the point $x_0$. One can identify the fundamental groups  $\pi_1 (H (D,m),h_0)$ and $F^m(D, x_0)$.

By  definition of  $D_2$ the function $y$ never equals to zero or to infinity on $H(D,m)$. Hence $y$  defines a map $y:H(D,m)\rightarrow \Bbb C\setminus \{ 0\}$.    Let $y_*:\pi_1 (H (D,m),h_0)\rightarrow \pi_1(\Bbb C\setminus \{0\}, y(h_0))$ be the induced homomorphism of the fundamental groups. The group $\pi_1 (H (D,m),h_0)$ is identified with the group $F^m(D,x_0)$ and the group $\pi_1(\Bbb C\setminus \{0\}, y(h_0))$ is isomorphic to $\Bbb Z$. So $\ker y_*\subset F^{m+1}(D,x_0)$. Thus all loops from the group $y_*(F^{m+1}(D,x_0))$ do not wind around the origin $0\in \Bbb C$.  Hence any germ of  $y^{1/q}$ does not change its value after continuation along a loop from the group $F^{m+1}(D,x_0)$. So $y^{1/q}$ is a single- valued function on $H(D,m+1)$. The proof is completed since $y^{1/q}$ is an algebraic function.
\end{proof}

\begin{lem}\label{lem2.14} An algebraic function $y$ is an $R$-function if and only if its monodromy group  is solvable.
\end{lem}
\begin{proof} Assume that $y$ is defined by (\ref{eq2.1}). Let $D$ be the hypersurface $P_n J=0$ where $P_n$ is the leading coefficient and $J$ is the discriminant of (\ref{eq2.1}). Let  $M$ be  the monodromy group of $y$. Consider   the natural homomorphism $p:F(D,x_0)\rightarrow M$. If $M$ is solvable then for some natural number $m$ the $m$-th commutator of $M$ is the identity element $e$. The function $y$ becomes  single-valued on the covering $H(D,m)$ since $F^m(D,x_0)\subset p^{-1}(M^m)=p^{-1}(e)$. Conversely,  if $y$ is a single-valued function on some covering $H(D,m)$ then $p(F^m(D,x_0))=e$. But $p(F^m(D,x_0))=M^m$. Thus the monodrogy group $M$ is solvable.
\end{proof}

\begin{thm}\label{thm2.15} If an algebraic function  has unsolvable monodromy group
then it cannot be represented by compositions of rational functions and radicals
\end{thm}

\begin{proof} Lemma~\ref{lem2.13} and Lemma~\ref{lem2.14} show that the class of $R$-functions is closed under arithmetic operations and compositions with radicals. Lemma~\ref{lem2.14}  shows that the monodromy group of any $R$-function is solvable.
\end{proof}

\subsubsection{Compositions of analytic functions and radials}\label{subsecss2.5}
In this section   we describe a class of multivalued functions in a domain $U\subset \Bbb C^N$
 representable by composition of single-valued analytic functions and radicals.

 A multivalued function $y$ in $U$  is called an {\it algebroidal function} in $U$ if it satisfies an irreducible equation
\begin{eqnarray}
y^n + f_{n-1}y^{n-1} + \dots +f_0=0 \label{eq2.2}
\end{eqnarray}
whose coefficients $f_{n-1},\dots, f_0$ are analytic functions in $U$. An algebroidal function could be considered as a continuous multivalued function in $U$ which has finitely many values.
\begin{thm}\label{thm2.16} (\cite{[3]},~\cite{[4]})A multivalued  function $y$ in the domain $U$   can be  represented by composition of radicals and single valued analytic functions  if and only  $y$ is an algebroidal  function in $U$ with solvable monodromy group.
\end{thm}

To prove the ``only if" part one can repeat the proof of Theorem~\ref{thm2.15} replacing coverings over domains $\Bbb C^N\setminus D$ by coverings over domains $U\setminus \tilde D$ where  $\tilde D$ is an analytic hypersurface in $U$.

  To prove Theorem~\ref{thm2.16} in the opposite direction one can use Theorem~\ref{thm2.8} in the same way as it was used in the proof of Theorem~\ref{thm2.10}.

\subsubsection{Local representability} In this section   we describe a  a local version of Theorem~\ref{thm2.16}.

Let $y$ be an algebroidal function in $U$ defined by (\ref{eq2.2}).  One can localize the equation (\ref{eq2.2})  at any point $p\in U$, i.e. one can replaced the coefficients $f_i$ of the equation (\ref{eq2.2})  by their germs at $p$. After such a localization  the equation (\ref{eq2.2}) can became reducible, i.e. it can became representable as a  product of irreducible equations. Thus an algebroidal functions  $y$ in  arbitrary small neighborhood of a point $p$ defines several  algebroidal functions, which we will call {\it ramified germs of $y$ at $p$.} For a ramified germ of $y$ at $p$ the monodromy group is defined (as the monodromy group of an algebroidal function in an arbitrary small  neighborhood of the point $p$).

A ramified germ of  an algebroidal function $y$   of one variable $x$ in a neighborhood of a  point $p\in \Bbb C^1$ has a simple structure: its monodromy group is a cyclic group $\Bbb Z/m  \Bbb Z$  and it  can be represented as a composition of a radical and an analytic single-valued function: $y(x)=f( (x-p)^{1/m}))$ where $m$ is the  ramification order of $y$. The following corollary follows from Theorem 16.
\begin{cor}\label{cor2.17}{ (\cite{[3]},~\cite{[4]})} 1) If a multivalued  function $y$ in the domain $U$   can be  represented by composition of an algebroidal functions of one variable  and single valued analytic functions  then the monodromy group of any ramified germ of $y$  is solvable.

2) If the monodromy group of a ramification germ of $y$ at  $p$ is solvable then in a small neighborhood of $p$ it can be represented by composition of radicals and single valued analytic functions.
\end{cor}

The {\it local monodromy group} of  an algebroidal function $y$ at a point $p\in U$ is the monodromy group of the equation (\ref{eq2.2}) in an arbitrary small neighborhood of the point $p$. The ramified germs of $y$ at the point $p$  correspond to the orbits of the local monodromy group  actions. This statement can be proven in the same way as Lemma~\ref{lem2.4} was proved.

\subsubsection{Application to the 13-th Hilbert problem}

In 1957 A.N.~Kolmogorov and V.I.~Arnold proved the following totally unexpected theorem which gave a negative solution to the 13-th Hilbert problem.

\begin{thm}\label{thmKA}{\it (Kolmogorov--Arnold)} Any continuous function of $n$ variables can be represented as the
composition of functions of a single variable with the help of addition.
\end{thm}

The 13-th Hilbert problem has the following  algebraic version which still remains open:
{\it Is it possible to represent any algebraic function of $n>1$ variables by algebraic functions of a smaller number of variables   with the help of composition and arithmetic operations?}

An {\it entire algebraic function} $y$ in $\Bbb C^N$ is an algebraic  function defined in $U=\Bbb C^N$ by  an equation (\ref{eq2.2}) whose coefficient $f_i$ are polynomials.  An entire algebraic function could be considered as a continuous algebraic function.

It turns out that in Kolmogorov--Arnold Theorem one cannot replace  continuous  functions by  entire algebraic functions.

\begin{thm}\label{thm2.18}(\cite{[3]},~\cite{[4]})  If an  entire algebraic function can be represented as a composition of polynomials and entire algebraic functions of one variable, then its local monodromy group at each point is solvable.
\end{thm}

\begin{proof} Theorem~\ref{thm2.18} follows from  from Corollary~\ref{cor2.17}.
\end{proof}

\begin{cor}\label{cor2.19} A function $y (a,b)$, defined by equation $ y^5 +ay+b=0,$ cannot be expressed in terms of entire algebraic functions of a single variable by means of composition, addition and multiplication.
\end{cor}

\begin{proof} Indeed, it is easy to check that the local monodromy group of $y$ at the origin is the unsolvable permutation group $S_5$ (see~\cite{[3]},~\cite{[4]}).
\end{proof}

Division is not a continuous operation and it destroys the locality. One cannot add  division to the operations used in Theorem~\ref{thm2.18}. It is easy to see that the function $y(a,b)$ from Corollary~\ref{cor2.17} can be expressed in terms of  entire algebraic functions of a single variable by means of composition  and arithmetic operations: $y(a,b)=g({b}/\root4\of{a^5})\root4\of a$, where $g(u)$ is defined by equation
$g^5+g+u=0$.

The following particular case of the algebraic version of the 13-th Hilbert problem still remains open.\\

\noindent{\it Problem.} Show that there is an algebraic function   of two variables which  cannot be expressed in terms of  algebraic functions of a single variable by means of composition  and arithmetic operations.

\subsection{One Dimensional Topological Galois Theory}\label{ch2sec2}

\subsubsection{Introduction}
In the one-dimensional version we consider functions from Liouvillian classes  as multi-valued analytic functions of one complex variable.
It turns out that there exist topological restrictions on the way the
Riemann surface of a function from a certain Liouvillian class   can be positioned over the complex plane. If a function does not satisfy these restrictions, then it cannot  belong to the corresponding Liouvillian class.

Besides a geometric appeal, this approach has the following
advantage. Topological obstructions relate to
branching. It turns out that if a function does not belong to a certain Liouvillian class by topological reasons then it automatically does not belong to a much wider class of functions.
This wider class can be obtained if we
add to the Liouvillian  class all single valued functions having at most a countable set of singularities   and allow them to enter all formulas.

The composition of functions is not
an algebraic operation. In differential algebra, this operation
is replaced with a differential equation describing it.
However, for example, the Euler $\Gamma$-function does not satisfy any
algebraic differential equation. Hence it is pointless to look
for an equation satisfied by, say, the function $\Gamma(\exp\,
x)$  and one can not describe it algebraically  (but the function $y=\exp (\Gamma(x))$ satisfies the equation $y'=\Gamma' y$ over a differential field containing $\Gamma$ and it makes sense in the differential algebra).  The only known results  on non-representability of
functions by quadratures and, say, the Euler $\Gamma$-functions
are obtained by our method.

On the other hand, our method cannot be used to prove that a
particular single valued meromorphic function does  not belong to a certain Liouvillian class.

There are the following topological obstructions to representability of functions by
generalized quadratures, $k$-quadratures and quadratures (see section~\ref{ssec2.2.6}).

Firstly, the functions representable by generalized quadratures and, in particular, the
functions representable by $k$-quadratures and quadratures  may have no more than
countably many singular points in the complex plane (see section~\ref{ssec2.2.4}).

Secondly, the monodromy group of a function representable by quadratures is
necessarily solvable (see section~\ref{ssec2.2.6}). There are similar restrictions for for a function representable by generalized quadratures and $k$-quadratures.
However, these restrictions are more involved.
To state them, the monodromy group should be regarded not as an abstract group
but rather as a transitive subgroup in the permutation group.
In other terms, these restrictions make use not only of the monodromy group
but rather of the {\it monodromy pair} of the function consisting of
the monodromy group and the stabilizer of some germ of the function
(see section~\ref{ssec2.2.7}).

\smallskip

One can prove that the only
reasons for unsolvability in finite terms of  Fuchsian linear
differential equations are topological (see section~\ref{ssec2.2.12}).In other words, if there are no topological
obstructions to solvability of a Fuchsian equation by  generalized quadratures (by $k$-quadratures, by quadratures), then this equation is solvable by  generalized quadratures (by $k$-quadratures or  by quadratures respectively).
The proof  is based on a linear-algebraic part of differential Galois theory (dealing with linear algebraic groups and their differential invariants),

\subsubsection{Solvability  in finite terms and Liouvillian classes of functions}

An equation is solvable ``in finite terms'' (or is solvable ``explicitly'') if its solutions belong to a certain class of functions. Different classes of functions correspond to different notions of  solvability in finite terms.

A class of functions can be introduced by specifying a list of {\it basic functions}
and a list of {\it admissible operations} (see section~\ref{classes}).
Given the two lists, the class of functions is defined as the set
of all functions that can be obtained
from the basic functions by repeated application of admissible operations.
(see section~\ref{classes}) Below, we define  classes of functions  in exactly this way.

Liouvillian classes of functions, which appear in the problems of integrability in finite terms,
contain multivalued functions. Thus  the operations on multivalued functions have to be defined. Such a definition can be found in  section~\ref{classes} (note that in the multidimensional case we use a slightly different, more restricted definition of the operations on multivalued functions).

We  need the list of {\it basic elementary functions} (see section~\ref{elemclasses}).
In essence, this list contains functions that are studied in high-school and which
are frequently used in pocket calculators .

We also use  the {\it classical operations}  on functions, such as the arithmetic operations, the operation of composition and so on.  The list of such operations is presented in section~\ref{elemclasses} .

We can now return to the definition of Liouvillian classes of single variable functions.

\paragraph{Functions  representable by radicals}
List of basic functions: all complex constants, an independent variable $x$.
List of admissible operations:
arithmetic operations and the operation of taking the $n$-th root
$f^{\frac {1}{n}}$, $n=2,3,\dots$, of a given function~$f$.

\paragraph{Functions  representable by $k$-radicals}
List of basic functions: all complex constants, an independent variable $x$.
List of admissible operations:
arithmetic operations and the operation of taking the $n$-th root
$f^{\frac {1}{n}}$, $n=2,3,\dots$, of a given function $f$, the operation of solving algebraic equations of degree $\leq k$

\paragraph{Elementary functions}
List of basic functions: basic elementary functions.
List of admissible operations: compositions, arithmetic operations, differentiation.

\paragraph{Generalized elementary functions}
This class of functions is defined in the same way as the class of elementary functions.
We only need to add the operation of solving algebraic equations to the list
of admissible operations.

\paragraph{Functions  representable by quadratures}
List of basic functions: basic elementary functions.
List of admissible operations: compositions, arithmetic operations, differentiation, integration.

\paragraph{Functions  representable by $k$-quadratures}
This class of functions is defined in the same way as the class of functions representable by quadratures.
We only need to add the operation of solving algebraic equations of degree
at most $k$ to the list of admissible operations.

\paragraph{Functions  representable by generalized quadratures}
This class of functions is defined in the same way as
the class of functions representable by quadratures.
We only need to add the operation of solving algebraic equations to the
list of admissible operations.

\subsubsection{Simple formulas with complicated topology}

Developing topological Galois theory I followed the following plan:
\medskip

I. To find a wide class of multivalued functions such that:
\smallskip

a) it is closed under all classical operations;

b) it contains all entire functions and all functions from each Liouvillian class;

c) for functions from the class the monodromy group  is well defined.
\medskip

II. To use the monodromy group instead of the Galois group inside the class.

\medskip

Let us discuss some difficulties that one need to overcome on this
way.

\begin{ex} Consider an elementary function $f$ defined by the
following formula:
$$
f(z)=\ln \sum\limits_{j=1}^n\lambda_j\ln (z-a_j)
$$
where $a_{j}$  are different points in the
complex line, and $\lambda_j\in \Bbb C$ are constants.
\end{ex}

Let $\Lambda$ denote the additive subgroup of
complex numbers generated by the constants $\lambda_1$,
$\dots$, $\lambda_n$. It is clear that if $n>2$, then for
almost every collection of constants $\lambda_1,\dots,
\lambda_n$, the group $\Lambda$ is everywhere dense in the
complex line.

\begin{lem}\label{lem2.20}
  If the group $\Lambda$ is dense in the complex line, then the elementary function $f$
has a dense set of logarithmic ramification points.
\end{lem}

\begin{proof}
  Let $g$ be the multivalued function defined by the formula
$$g(z)=\sum\limits_{j=1}^n\lambda_j\ln (z-a_j).$$ Take  a point
$a\neq a_j$, $j=1,\dots, n$ and let $g_a$ be one of the germs of $g$ at $a$. A loop around the points
$a_1,\dots, a_n$ adds the number $2\pi i\lambda$ to the germ
$g_a$, where $\lambda$ is an element of the group $\Lambda$.
Conversely, every germ $g_a+2\pi i\lambda$, where $\lambda \in
\Lambda$, can be obtained from the germ $g_a$ by the analytic
continuation along some loop. Let $U$ be a small neighborhood
of the point $a$, such that the germ $g_a$ has  a singe-valued analytic continuation $G$  on $U$.  The image $V$
of the domain $U$ under the map $G: U\to\Bbb C$ is open.
Therefore, in the domain $V$, there is a point of the form
$2\pi i \lambda$, where $\lambda \in \Lambda$. The function
$G-2\pi i \lambda$ is one of the branches of the function $g$
over the domain $U$, and the zero set of this branch in the
domain $U$ is nonempty. Hence, one of the branches of the
function $f=\ln g$ has a logarithmic ramification point in $U$.
\end{proof}

The set $\Sigma$ of singular points of  the function $f$ is  a {\it countable set} (see   section~\ref{ssec2.2.4}).  Under assumptions of Lemma~\ref{lem2.20} the set $\Sigma$ is  everywhere dense.

It is not hard to verify that  the monodromy group (see section~\ref{ssec2.2.7})  of the function $f$ has the cardinality of the continuum.  This is not surprising: the
fundamental group $\pi_1(\Bbb C \setminus \Sigma)$ has obviously the
cardinality of the continuum provided that $\Sigma$ is a countable
dense set.

One can also prove that the image of the fundamental group $\pi
_1(\Bbb C\setminus \{\Sigma \cup b\})$ of the complement of the set
$\Sigma\cup b$, where $b\not\in \Sigma$, in the permutation group  is a proper subgroup of the
monodromy group of $f$.

The fact that the removal of one extra point can change the monodromy group, makes all proofs more complicated.

{\it Thus even simplest elementary functions can have dense singular
sets and monodromy groups of cardinality of the continuum. In addition the removal of an extra point can change their monodromy groups.}

\subsubsection{ Class of $\mathcal S$-functions}\label{ssec2.2.4}

In this section, we define a broad class of functions of one
complex variable needed in the construction of  topological  Galois theory.

\begin{defn} A multivalued analytic function of one complex
variable is called a
$\mathcal S$-function, if the set
of its singular points is at most countable.
\end{defn}
Let us make this definition more precise. Two regular germs $f_a$ and $g_b$
defined at points $a$ and $b$ of the Riemann sphere $\Bbb S^2$ are
called {\it equivalent}
 if the germ $g_b$ is obtained from the
germ $f_a$ by the analytic continuation along some
path. Each germ $g_b$ equivalent to the germ $f_a$ is also
called a regular germ of the multivalued analytic function $f$
generated by the germ $f_a$.

A point $b\in \Bbb S^2$ is said to be a
{\it singular point}
 for the germ $f_a$ if there exists a path
$\gamma: [0,1]\to \Bbb S^2$, $\gamma (0)=a$, $\gamma (\ref{eq2.1})=b$ such
that the germ has no analytic continuation along this path, but
for any $\tau$, $0\leq \tau <1$, it admits an analytic continuation
along the truncated path $\gamma:[0,\tau]\rightarrow S^2$.

It is easy to see that equivalent germs have the same set of singular
points. A regular germ is called a {\it $\mathcal S$-germ}, if
the set of its singular points is at most countable. A
multivalued analytic function is called a $\mathcal S$-function if each of its regular germ is a $\mathcal S$-germ.

\begin{thm}\label{thm2.21} (on stability of the class of $\mathcal S$-functions)
  The class $\mathcal S$ of all $\mathcal S$-functions is stable under the following
operations:\begin{enumerate}
\item differentiation, i.~e. if $f\in \mathcal S$, then $f'\in
    \mathcal S$;
\item  integration, i.~e. if $f\in \mathcal S$ and $g'=f$,
    then $g\in \mathcal S$;
\item composition, i.~e. if $g,\ f\in \mathcal S$, then
    $g\circ f\in \mathcal S$;
\item meromorphic operations, i.~e. if $f_i\in \mathcal S$,
    $i=1$, $\dots$, $n$, the function $F(x_1, \dots,x_n)$
    is a meromorphic function of $n$ variables, and
    $f=F(f_1,\dots ,f_n)$, then $f\in \mathcal S$;
\item  solving algebraic equations, i.~e. if $f_i\in
    \mathcal S,\ i=1, \dots ,n$, and $f^n+f_1f^{n-1}+\dots
    +f_n=0$, then $f\in \mathcal S$;
\item  solving linear differential equations, i.e. if
    $f_i\in \mathcal S,\ i=1,\dots ,n$, and
    $f^{(n)}+f_1f^{(n-1)}+\dots +f_nf=0$, then $f\in
    \mathcal S$.
\end{enumerate}
\end{thm}

\begin{remark}Arithmetic operations and the
exponentiation are examples of meromorphic operations, hence
{\it the class of $\mathcal S$-functions is stable under the
arithmetic operations and the exponentiation.}
\end{remark}

\begin{cor} \label{cor2.22}(see~\cite{[6]})
  If a multivalued function $f$ can be obtained from single
valued $\mathcal S$-functions by integration, differentiation,
meromorphic operations, compositions, solutions of algebraic
equations and linear differential equations, then the function
$f$ has at most countable number of singular points.
\end{cor}

\begin{cor}\label{cor2.23}  A function having uncountably many singular points
cannot be represent  by generalized quadratures. In particular it cannot be a generalized  elementary function and it cannot be represented by $k$-quadratures or by quadratures.
\end{cor}

\begin{ex} Consider a discrete group $\Gamma$ of fractional linear transformations  of the open unit ball $U$ having a compact fundamental domain. Let $f$ be a nonconstant  meromorphic function on $U$ invariant under the action of $\Gamma$. Each point on the boundary $
\partial U$ belongs to the closure of the set of poles  of $f$, thus the set  $\Sigma$ of singular points of $f$ contains $\partial U$. So $\Sigma$ has the cardinality of the continuum and $f$ cannot be expressed by generalized quadratures.
\end{ex}

\subsubsection{Monodromy group of a $\mathcal S$-function}

The  {\it  monodromy group } of a $\mathcal{S}$-function $f$  is the group of all  permutations of
the sheets of the Riemann surface of $f$  which are induced by
motions  around the singular set $\Sigma $ of the function $f$. Below we discuss this definition more precisely.

Let $F_{x_0}$ be the set of all  germs of the
$\mathcal S$-function $f$ at point $x_0\notin \Sigma$. Consider a closed curve  $\gamma $ in
$\Bbb S^2\setminus \Sigma $ beginning and ending at the point $x_0$. Given a germ $y\in F_{x_0}$ we can continue it along the loop $\gamma$ to obtain another germ $y_\gamma\in Y_{x_0}$.
Thus each such loop $\gamma$   corresponds to a permutation $S_{\gamma}:F_{x_0}\rightarrow F_{x_0}$ of the set $F_{x_0}$ that maps a germ $y\in F_{x_0}$ to the germ $y_{\gamma}\in F_{x_0}$.

It is  easy to see that the map $\gamma\rightarrow S_\gamma$ defines a homomorphism   from the fundamental
group $\pi _1(\Bbb S^2\setminus \Sigma,x_0)$ of the domain $\Bbb S^2\setminus \Sigma $ with the base point $x_0$ to the group $S(F_{x_0})$ of permutations. The {\it
monodromy group} of the $\mathcal S$-function $f$  is the image of
the fundamental group   in the group
$S(F_{x_0})$ under this homomorphism.

\begin{remark} Instead of the point $x_0$ one can choose any other point $x_1\in \Bbb S^2\setminus \Sigma$. Such a choice will not change the    monodromy group  up to an isomorphism. To fix this isomorphism one can choose any curve $\gamma:I\rightarrow \Bbb C^N\setminus \Sigma$ where $I$ is the segment $ 0\leq t\leq 1$ and $\gamma(0)=x_0$, $\gamma(1)=x_1$ and identify each germ $f_{x_0}$ of $f$   with its continuation $f_{x_1}$ along $\gamma$.
\end{remark}

\subsubsection{Strong non  representability  by quadratures}\label{ssec2.2.6}
One can prove the following  theorem.

\begin{thm}\label{thm2.24} (see~\cite{[6]})  The class of all
$\mathcal S$-functions, having a solvable monodromy group, is stable
under composition, meromorphic
operations,  integration and   differentiation.
\end{thm}

\begin{defn} A function $f$ is {\it strongly nonrepresentable by quadratures }  if it does not belong to a class of functions defined by the following data. List of basic functions: basic elementary functions and all single valued $\mathcal S$-function.
List of admissible operations: compositions, meromorphic operations, differentiation and integration.
\end{defn}

Theorem~\ref{thm2.24} implies  the  following corollary.

\begin{cor}[Result on  quadratures] If the monodromy group of an $S$-function $f$ is not solvable, then $f$ is   strongly non representable by quadratures.
\end{cor}

\begin{ex} The monodromy group of an algebraic function $y(x)$ defined by an equation $y^5+y -x=0$ is the unsolvable group $S_5$. Thus $y(x)$ provides an example of a function with  finite set of singular points, which is strongly non representable by quadratures.
\end{ex}

The following Corollary  contains  a  stronger result on non representability of algebraic functions by quadratures.
\begin{cor}\label{cor2.25}  If an algebraic function of one complex variable  has unsolvable monodromy group then it is strongly non representable by quadratures.
\end{cor}

For algebraic functions of several complex variables there is a result  similar   to Corollary~\ref{cor2.25}

\subsubsection{The  monodromy pair}\label{ssec2.2.7}

The  monodromy group of a function $f$ is not only an abstract group
but is also a transitive group of permutations  of  germs of $f$ at a non singular point $x_0$.

\begin{defn} The {\it monodromy  pair of an $\mathcal S$-function $f$} is a
pair of groups, consisting of the monodromy group  of $f$ at  $x_0$
 and the stationary subgroup  of a certain germ   of $f$ at $x_0$.
\end{defn}

The  monodromy pair is well defined, i.e. this pair
of groups, up to  isomorphisms, does not  depend on the  choice of the non singular point  and on the  choice of the germ of $f$ at this point. The intersection of the stationary subgroups of all germs of $f$ at $x_0$ is the identity element since the monodromy group acts transitively on this set.

 \begin{defn} A pair of groups $[\Gamma,\Gamma_0]$ is an {\it almost normal pair} if there is a finite subset $A \subset \Gamma$ such that the intersection $\bigcap_{a\in A} a(\Gamma_0)a^{-1}$ is equal to the identity element.\end{defn}

\begin{defn} The pair of groups $[\Gamma,\Gamma_0]$ is called
an {\it almost solvable pair of groups} if there exists a sequence
of subgroups $$ \Gamma=\Gamma_1\supseteq\dots \supseteq \Gamma_m,
\quad \Gamma_m \subset \Gamma_0, $$ such that  for every   $i,
1\leq i\leq m-1$ group $\Gamma _{i+1}$ is a  normal divisor of
group  $\Gamma _i$ and the factor group  $\Gamma _i/\Gamma
_{i+1}$ is either a commutative group, or a finite group.
\end{defn}

\begin{defn} The pair of groups $[\Gamma,\Gamma_0]$ is called
a {\it $k$-solvable pair of groups} if there exists a sequence
of subgroups $$ \Gamma=\Gamma_1\supseteq\dots \supseteq \Gamma_m,
\quad \Gamma_m \subset \Gamma_0, $$ such that  for every   $i,
1\leq i\leq m-1$ group $\Gamma _{i+1}$ is a  normal divisor of
group  $\Gamma _i$ and the factor group $\Gamma _i/\Gamma
_{i+1}$ is either a commutative group, or a subgroup of the group $S_k$ of permutations of $k$ elements.
\end{defn}

We say that  group $\Gamma $ is
{\it almost solvable} or {\it $k$-solvable} if pair $[\Gamma ,e]$, where $e$ is the  group
containing only the unit element, is almost solvable or $k$-solvable respectively.

{ It is easy to see that {\it an almost normal pair of groups $[\Gamma,\Gamma_0]$ is almost solvable or $k$-solvable if and only if the group $\Gamma$ is almost solvable or $k$-solvable respectively}}.

\subsubsection{Strong non-representability  by $k$-quadratures}\label{ssec2.2.8}
One can prove the following  theorem.
\begin{thm}\label{thm2.26} (see~\cite{[6]})  The class of all
$\mathcal S$-functions, having a $k$-solvable monodromy pair, is stable
under composition, meromorphic
operations,  integration, differentiation and solutions of algebraic equations of degree $\leq k$.
\end{thm}

\begin{defn}  A function $f$ is {\it strongly  non representable by $k$-quadratures}  if it does not belong to a class of functions defined by the following data. List of basic functions: basic elementary functions and all single valued $\mathcal S$-function.
List of admissible operations: compositions, meromorphic operations, differentiation, integration and solutions of algebraic equations of degree $\leq k$.
\end{defn}
Theorem~\ref{thm2.26} implies  the  following corollary.

\begin{cor}[Result on  $k$-quadratures] If the monodromy pair of an $S$-function $f$ is not $k$-solvable, then $f$ is strongly  non representable by $k$-quadratures.
\end{cor}

\begin{ex} The monodromy group of an algebraic function $y(x)$ defined by an equation $y^n+y -x=0$ is the permutation group group $S_n$. For $n\geq 5$ the group $S_n$ is not an $(n-1)$-solvable group. Thus $y(x)$ provides an example of a function with  finite set of singular points which is strongly non representable  by $(n-1)$-quadratures.
\end{ex}

This example can be  generalized.
\begin{cor}\label{cor2.27} (see~\cite{[6]}) If an algebraic function of one complex variable  has non $k$-solvable monodromy group then it is strongly non representable by $k$-quadratures.
\end{cor}

\begin{thm}\label{thm2.28}(see~\cite{[6]})An  algebraic function of one variable whose monodromy group is $k$-solvable, can be represented by $k$-radicals.
\end{thm}
Results similar to Corollary~\ref{cor2.27} and Theorem~\ref{thm2.28} hold also for algebraic functions of several complex variables.

\subsubsection{Strong non-representability  by generalized quadratures}\label{ssec2.2.9}
One can prove the following  theorem.
\begin{thm}\label{thm2.29} (see~\cite{[6]})   The class of all
$\mathcal S$-functions, having an almost solvable monodromy pair, is stable
under composition, meromorphic
operations,  integration, differentiation and solutions of algebraic equations.
\end{thm}

\begin{defn}A function $f$ is {\it strongly  non representable by generalized quadratures}  if it does not belong to a class of functions defined by the following data. List of basic functions: basic elementary functions and all single valued $\mathcal S$-function.
List of admissible operations: compositions, meromorphic operations, differentiation, integration and solutions of algebraic equations.
\end{defn}

Theorem~\ref{thm2.29} implies  the  following corollary.

\begin{cor}[Result on  generalized quadratures] If the monodromy pair of an $S$-function $f$ is not almost solvable, then $f$ is strongly  non representable by generalized quadratures.
\end{cor}

Suppose that the Riemann surface of a
function $f$ is a universal covering space over the Riemann sphere with $n$ punched points. If $n\geq 3$ then the  function $f$ is strongly non representable
by generalized quadratures. Indeed,
the monodromy pair of  $f$ consists of the free  group with $n-1$ generators, and its unit subgroup. It is easy to see that
such a pair of groups is not almost solvable.

\begin{ex} Consider the function $z(x)$, which maps  the upper half-plane onto a triangle with vanishing angles, bounded by three circular arcs.
The Riemann surface of  $z(x)$ is a universal covering space over the
sphere with three punched points\footnote {it is easy to see that the function $z(x)$ maps its Riemann surface to the open ball whose boundary contains the vertices of the triangle. These properties of the function $z(x)$ play the crucial role in  Picard's beautiful proof of his Little Picard Theorem.}.  Thus $z(x)$ is strongly non
representable by generalized quadratures.
\end{ex}

\begin{ex}
Let $K_1$ and $K_2$ be the following elliptic
integrals, considered as the functions of the parameter $x$:
  $$
K_1(x)=\int_0^1\frac{dt}{\sqrt{(1-t^2)(1-t^2x^2)}}\ {\hbox{ and }}
K_2(x)=\int_0^{\frac{1}{x}}\frac{dt}{\sqrt{(1-t^2)(1-t^2x^2)}}. $$
The functions  $z(x)$ can be obtained
from $K_1(x)$ and from $K_2(x)$  by quadratures. Thus both functions
 $K_1(x)$ and $K_2(x)$ are strongly non representable  by generalized quadratures.
\end{ex}

In the next section  we will list all polygons $G$ bounded by circular arcs for which the Riemann map of the upper half-plan onto $G$ is representable by generalized quadratures.

\subsubsection{Maps of the upper half-plane onto a curved polygon}\label{ssec2.2.10}

Consider a  polygons $G$ on the complex plane  bounded by circle arcs, and the function $f_G$ establishing the Riemann mapping of the upper half-plane onto the polygon
$G$. The Riemann--Schwarz reflection
principle allows one to describe the monodromy group $L_G$ of the function $f_G$ and to show that all singularities of $f_G$ are simple enough. This information together with Theorem~\ref{thm2.29} provide a complete classification of all polygons $G$ for which the function $f_G$ is representable in explicit form (see~\cite{[6]}).

If a polygon $\tilde G$ is obtained from a polygon $G$ by a linear transformation $w:\Bbb C  \to \Bbb C$ then $f_{\tilde G}=w (f_G)$. Thus it is enough to classify $G$ up to a linear transformation.

\paragraph{The first case of integrability:}\hspace{-.15in} the continuations of all sides of the polygon $G$ intersect at one point.

Mapping this point to infinity by a fractional linear transformation, we obtain a polygon $G$ bounded by straight line segments.
All transformations in the group $L(G)$ have the form $z \to az+b.$ All germs of
the function $f = f_G$ at a non-singular point c are obtained from a fixed germ $f_c$ by
the action of the group $L(G)$ consisting of the affine transformations $z\to a z+b$. The germ $R_c = (f''/ f')_c$ is invariant under the action of the group $L(G)$. Therefore, the germ $R_c$ is a germ of a single valued function $R$.
The singular points  of  $R$ can only be poles (see~\cite{[6]}). Hence the function $R$ is  rational. The equation $ f''/ f'= R$ is
integrable by quadratures. This integrability case is well known. The function $f $ in
this case is called the {\it Christoffel--Schwarz integral}.

\paragraph{The second case of integrability:} \hspace{-.15in} there is  a pair of points such that, for every side of the polygon
$G$, these points are either symmetric with respect to this side or belong to the continuation
of the side.

We can map these two points to zero and infinity by a fractional
linear transformation. We obtain a polygon $G$ bounded by circle arcs centered at
point 0 and intervals of straight rays emanating from 0 (see~\cite{[6]}). All transformations
in the group $L(G)$ have the form $z\to az, z \to b/z.$  All germs of the function
$f = f_G$ at a non-singular point $c$ are obtained from a fixed germ $f_ c $ by the action of
the group $L(G)$ :$$f_ c \to a  f_c, f_ c \to b/ f _c.$$ The germ $R_c = (f'_ c/ f_ c)^2$ is invariant under the action of the group $L(G)$. Therefore, the germ $R_c$ is a germ of a single valued function $R$. The singular points  of  $R$ can only be poles (see~\cite{[6]}). Hence the function $R$ is  rational. The equation $R = (f'/ f)^2$ is integrable by quadratures
\bigskip

\paragraph{The finite nets of  circles.}\hspace{-.15in} To describe the third  case of integrability we need to define first the finite net of circles on the complex plane. The classification of finite groups, generated by reflections in the Euclidian space $\Bbb R^3$ is well known. Each such group is the symmetry  group of the following bodies:
\begin{enumerate}
\item a regular $n$-gonal pyramid,
\item a regular $n$-gonal diheron, or the body formed by two equal regular n-gonal pyramidssharing the base,
\item  a regular tetrahedron,
\item  a regular cube or icosahedron,
\item  a regular dodecahedron or icosahedron.
\end{enumerate}
All these groups of isometries, except for the group of dodecahedron or icosahedron,
are solvable.

The intersections of the unit sphere, whose center coincides with the
barycenter of the body, with the mirrors, in which the body is symmetric, is a certain
net of great circles.  Stereographic projections of each of them is a net net of circles on complex plane defined up to a fractional linear transformation.  The nets corresponding to the bodies listed above will be called
the finite nets of  circles.
\bigskip
\paragraph{The third case of integrability:}\hspace{-.15in}  every side side of a polygon $G$  belongs to some finite net  of  circles. In this case  the function $f_G$ has finitely many branches. Since all singularities of
the function $f_G$ are algebraic (see~\cite{[6]}), the
function $f_G$ is an algebraic function. For all finite nets but the net  of dodecahedron or icosahedron, the algebraic function $f_G$ is representable by radicals. For the net  of dodecahedron or icosahedron the function $f_G$ is representable by radicals and solutions of degree five algebraic equations (in other words $f_G$ is representable by $k$-radicals).
\medskip

\paragraph{The strong non-representability.}\hspace{-.15in} Our results  imply the following:

\begin{thm}\label{thm2.30}(see~\cite{[6]}) If a  polygon $G$ bounded by circles arcs  does not belong to one of
the three  cases described  above, then the function $f_G $ is strongly non  representable by generalized quadratures.
\end{thm}

\subsubsection{Non-solvability of linear differential equations}\label{ssec2.2.11}
Consider a homogeneous linear differential equation
\begin{eqnarray}
y^{(n)}+ r_1y^{(n-1)}+\dots +r_ny=0,\label{eq2.3}
\end{eqnarray}
whose coefficients  $r_i$'s are rational functions  of the complex variable
$x$.  The set  $\Sigma\subset \Bbb C$  of poles of  $r_i$'s is called  {\it the set of
singular points}   of the equation (\ref{eq2.3}). At a  point $x_0\in \Bbb C\setminus \Sigma$ the germs of solutions of (\ref{eq2.3}) form a $\Bbb C$-linear  space $V_{x_0}$ of dimension $n$. The {\it monodromy group $M$ of the equation (\ref{eq2.3})} is the group of all linear transformations  of the space $V_{x_0}$  which are induced by motions around the set $\Sigma$. Below we discuss this
definition more precisely.

Consider a closed curve  $\gamma $ in
$\Bbb C\setminus \Sigma $ beginning and ending at the point $x_0$. Given a germ $y\in V_{x_0}$ we can continue it along the loop $\gamma$ to obtain another germ $y_\gamma\in V_{x_0}$.
Thus each such loop $\gamma$   corresponds to a map $M_{\gamma}:V_{x_0}\rightarrow V_{x_0}$
 of the space $V_{x_0}$to itself  that maps a germ $y\in V_{x_0}$ to the germ $y_{\gamma}\in V_{x_0}$. The map $M_\gamma$ is  linear  since an analytic continuation respects the arithmetic operations.
It is easy to see that the map $\gamma\rightarrow M_\gamma$ defines a homomorphism   of the fundamental
group $\pi _1(\Bbb C\setminus \Sigma,x_0)$ of the domain $\Bbb C\setminus \Sigma $ with the base point $x_0$ to the group $GL(n)$ of invertible linear transformations  of  the space  $V_{x_0}$.

The {\it
monodromy group} $M$ of the equation (\ref{eq2.3}) is the image of
the fundamental group   in the group
$GL(n))$ under this homomorphism.

\begin{remark} Instead of the point $x_0$ one can choose any other point $x_1\in \Bbb C\setminus \Sigma$. Such a choice will not change the    monodromy group  up to an isomorphism. To fix this isomorphism one can choose any curve $\gamma:I\rightarrow \Bbb C^N\setminus \Sigma$ where $I$ is the segment $ 0\leq t\leq 1$ and $\gamma(0)=x_0$, $\gamma(1)=x_1$ and identify each germ $y_{x_0}$ of solution of (\ref{eq2.3}) with its continuation $y_{x_1}$ along $\gamma$.
\end{remark}

\begin{lem}\label{lem2.31} The stationary subgroup in the monodromy group $M$ of the germ $y\in V_{x_0}$   of almost every solution  of the equation (\ref{eq2.3})  is trivial (i.e. contains only the unit element $e\in M$).
\end{lem}
\begin{proof}The monodromy group $M$ contains countable many linear transformations $M_i$. The space $L_i\subset V_{x_0}$ of fixed points of a non identity transformation $M_i$,  is a proper subspace of $ V_{x_0}$. The union $L$ of all subspaces $L_i$ is a measure zero subset of  $V_{x_0}$. The stationary subgroup in  $M$ of $y\in V_{x_0}\setminus L$  is trivial.
\end{proof}

\begin{thm}\label{thm2.32}(see~\cite{[6]})  If the monodromy group of the equation   (\ref{eq2.3})
is not almost solvable (is not $k$-solvable, or is not solvable)  then  almost every  solution of (\ref{eq2.3}) is strongly  non representable by generalized quadratures  (correspondingly, is strongly non representable by $k$-quadratures, or is strongly non representable by quadratures).
\end{thm}

Consider a homogeneous system of linear differential equations
\begin{eqnarray}
 y'=A y \label{eq2.4}
\end{eqnarray}
where $y=(y_1,\dots,y_n)$ is the unknown vector valued function and $A=\{a_{i,j}(x)\}$
is a $n\times n$ matrix, whose entries are rational functions of the
complex variable $x$. One can define the {\it monodromy group} of the equation (\ref{eq2.4}) exactly in the same way as it was defined for the equation (\ref{eq2.3}).

We will say that a vector valued function $y=(y_1,\dots, y_n)$  belongs to a certain class of functions if all of its components $y_i$  belong to this class. For example the statement "a vector valued function $y=(y_1,\dots,y_n)$ is strongly non representable by generalized quadratures" means that at least one component $y_i$ of $y$ is strongly non representable by generalized quadratures.

\begin{thm}\label{thm2.33}  If the monodromy group of the system  (\ref{eq2.4})
is not almost solvable (is not $k$-solvable, or is not solvable)  then almost every  solution  of (\ref{eq2.4})is strongly  non representable by generalized quadratures  (correspondingly, is strongly non representable by $k$-quadratures, or is strongly non representable by quadratures).
\end{thm}

\subsubsection{Solvability of  Fuchsian  equations}\label{ssec2.2.12}

The differential field  of rational functions of $x$ is isomorphic to the differential field $\mathcal R$ of germs of rational functions at the point $x_0\in \Bbb C\setminus \Sigma $.
Consider the differential field extension  $\mathcal R \langle y_1,\dots,y_n\rangle$ of $\mathcal R$ where the  germs  $ y_1,\dots y_n$ form a basis in the space $V_{x_0}$ of solutions of the equation~(\ref{eq2.3}) at $x_0$.

\begin{lem}\label{lem2.34} Every linear map $M_\gamma$ from the monodromy group of equation (\ref{eq2.3}), can be uniquely extended to a differential automorphism of the  differential field $\mathcal R\{ y_1,\dots,y_n\}$ over the field $\mathcal R$.
\end{lem}

\begin{proof}Every element $f\in \mathcal R \langle y_1,\dots,y_n\rangle$ is a rational function of the independent variable $x$, the germs of solutions $y_1,\dots,y_n$ and their derivatives. It can be continued meromorphically along  the curve
$\gamma\in \pi_1(\Bbb C\setminus \Sigma,x_0)$  together with  $y_1,\dots,y_n$.
This continuation gives the required differential
automorphism, since the continuation preserves the arithmetical
operations and differentiation, and every rational function of $x$ returns back
to its original  values (since it is a single-valued valued function). The differential automorphism is unique because the extension is generated by $y_1,\dots, y_n$.
\end{proof}

The {\it differential Galois group} (see~\cite{[6]},~\cite{[10]}) of the equation (\ref{eq2.3}) over $\mathcal R$ is the group of all differential automorphisms  of the  differential field $\mathcal R\{ y_1,\dots, y_n\}$ over the differential field  $\mathcal R$. According to Lemma 32 the monodromy group  of the equation (\ref{eq2.3}) can be considered as a subgroup of its differential Galois group  over $\mathcal R$.

The differential  field of  invariants of the monodromy group action is a subfield
of $\mathcal R \langle y_1,\dots,y_n\rangle$, consisting of the single-valued
functions. In contrast to  the algebraic case, in the case of  differential equations the field of invariants under the action of the
monodromy group can be bigger than the field of rational
functions. The reason is that  the solutions
of differential equations may grow exponentially in approaching
the singular points or infinity.

\begin{ex} All  solutions of the simplest  differential equation $y'=y$  are
single-valued exponential functions $y=C\exp x$, which are not rational.
\end{ex}

For the  wide class of   Fuchsian  linear
differential equations all the solutions,  while approaching
the singular points and the point infinity, grow polynomially.

The following  Frobenius theorem is an analog for Fuchsian  equations of  C.Jordan theorem (see~\cite{[6]}) for algebraic equations.

\begin{thm}[Frobenius]   For Fuchsian differential  equations
 the subfield of the differential field
$\mathcal R \langle y_1,\dots,y_n\rangle$, consisting of single-valued
functions, coincides with the field of rational functions.
\end{thm}

A system of linear differential equations (\ref{eq2.4}) is called a {\it Fuchsian system} if the matrix  $A$ has the  following form:
\begin{eqnarray}
 A(x)=\sum _{i=1}^k\dfrac{A_i}{x-a_i},\label{eq2.5}
\end{eqnarray}
where the $A_i$'s are constant matrices. Linear Fuchsian system of differential equations in many ways are similar to  linear Fuchsian  differential equations

In construction of explicit solutions of linear  differential equations the following theorem is needed.

\begin{thm}[(Lie--Kolchin)] Any  connected  solvable algebraic group acting by linear transformations on  a finite-dimensional vector space over  $\Bbb C$ is  triangularizable in a suitable basis.
\end{thm}

Using the Frobenius Therem and Lie--Kolchin Theorem  one can prove that the only
reasons for unsolvability of  Fuchsian linear
differential equations and systems of linear differential equations are topological. In other words, if there are no topological
obstructions to solvability then such equations and systems of equations are solvable. Indeed, the following theorems hold:

\begin{thm}\label{thm2.35}(see~\cite{[6]})  If the monodromy group of the linear Fuchsian differential  equation (\ref{eq2.3})
is almost solvable (is  $k$-solvable, or is solvable)  then   every  solution is  representable by generalized quadratures  (correspondingly, is  representable by $k$-quadratures, or is  representable by quadratures).

\end{thm}

\begin{thm}\label{thm2.36}(see~\cite{[6]})  If the monodromy group of the linear Fuchsian system differential equations (\ref{eq2.4})
is  almost solvable (is  $k$-solvable, or is  solvable)  then   every  solution is    representable by generalized quadratures  (correspondingly, is  representable by $k$-quadratures, or is  representable by quadratures).

\end{thm}

\subsubsection{Fuchsian systems with small coefficients}

In general the monodromy group of a given Fuchsian equation is very hard to compute.  It is known only for very special equations, including the famous hypergemetric equations. Thus Theorems~\ref{thm2.35} and~\ref{thm2.36} are not explicit.

If the matrix $A(x)$ in the system (\ref{eq2.4}) is triangular then one can easily  solve the system by quadratures. It turns out that if the matrix $A(x)$ has the form (\ref{eq2.5}), where the matrices $A_i$'s are sufficiently small, then the system (\ref{eq2.4}) with a non triangular matrix $A(x)$ is unsolvable by generalized quadratures for a topological reason.

\begin{thm}\label{thm2.37}(see~\cite{[6]}) If the  matrices $A_i$'s are sufficiently small, $\Vert
 A_i\Vert<\varepsilon(a_1,\dots,a_k,n)$, then the  monodromy group of the system
\begin{eqnarray}
 y'= (\sum _{i=1}^k\dfrac{A_i}{x-a_i})y \label{eq2.6}
 \end{eqnarray}
is almost solvable if and only if  the  matrices $A_i$'s are triangularizable in a suitable basis.
\end{thm}

\begin{cor}\label{cor38} If in the assumptions of Theorem 19 the  matrices $A_i$'s are not triangularizable in a suitable basis then almost every  solution of the system (6) is strongly non   representable by generalized quadratures.
\end{cor}

\subsubsection{Polynomials invertible by radicals}

In 1922 J.F.Ritt published (see~\cite{[12]}) the following beautiful theorem which fits nicely into topological Galois theory.
\begin{thm}(J.F. Ritt)
The inverse function of a polynomial with complex coefficients can be represented by radicals if and only if the polynomial is a composition of linear polynomials, the power polynomials $z\to z^n$, Chebyshev polynomials and polynomials of degree at most 4.
\end{thm}

{\sc Outline of proof  (following ~\cite{[2]})}

1)   {\it Every polynomial is a composition of primitive ones:} Every polynomial is a composition of polynomials that are not themselves compositions of polynomials of degree $>1$. Such polynomials are called {\it primitive}. Recall that a permutation group $G$ acting on a non-empty set $X$ is called {\it primitive} if $G$ acts transitively on $X$ and G preserves no nontrivial partition of $X$. {\it A polynomial is primitive if and only if the monodromy group of inverse of the polynomial   acts primitively on its branches}.
\smallskip

2)  {\it Reduction to the case of primitive polynomials:} A composition of polynomials is invertible by radicals if and only if each polynomial in the composition is invertible by radicals. Indeed, if each of the polynomials in composition is invertible by radicals, then their composition also is. Conversely, if a polynomial $R$ appears in the presentation of a polynomial $P$ as a composition $P=Q\circ R\circ S$ and $P^{-1}$ is representable by radicals, then $R^{-1}=S\circ P^{-1}\circ Q$ is also representable by radicals. Thus it is enough to classify only the primitive polynomials invertible by radicals.
\smallskip

3) {\it A result on solvable primitive permutation groups containing a full cycle:} A primitive polynomial is invertible by radicals if and only if the monodromy group of the inverse of the polynomial is solvable.  Since it acts primitively on its branches  and contains a full cycle (corresponding to a loop around the point at infinity on the Riemann sphere), the following group-theoretical result of Ritt is useful for the classification of polynomials invertible by radicals:

\begin{thm}\label{thmprim} (on primitive solvable groups with a cycle)
Let $G$ be a primitive solvable group of permutations of a finite set $X$ which contains a full cycle. Then either $|X|=4$, or $|X|$ is a prime number $p$ and $X$ can be identified with the elements of the field $F_p$ so that the action of $G$ gets identified with the action of the subgroup of the affine group $AGL_1(p)=\{x\to ax+b|a\in (F_p)^*,b\in F_p\}$ that contains all the shifts $x\to x+b$.
\end{thm}

4) {\it Solvable monodromy groups of the inverse of primitive polynomials:} It can be shown by applying the Riemann--Hurwitz formula that among the groups in Theorem~\ref{thmprim} on primitive solvable groups with a cycle,  only the following groups can be realized as monodromy groups of inverse of primitive polynomials: 1. $G\subset S_4$, 2. Cyclic group $G=\{x\to x+b\}\subset AGL_1(p)$, 3. Dihedral group $G=\{x\to \pm x+b\}\subset AGL_1(p)$.
\smallskip

5) {\it Description of primitive polynomials invertible by radicals:} It can be easily shown (see for instance [see~\cite{[12]},~\cite{[2]}) that the following result holds:

\begin{thm}\label{thm2.39}
If the monodromy group of  inverse of a primitive polynomial  is a subgroup of the group $\{x\to \pm x+b\}\subset AGL_1(p)$, then up to a linear change of variables the polynomial is either a power polynomial or a Chebyshev polynomial.
\end{thm}

Thus the polynomials whose inverse have monodromy groups 1-3 are respectively 1. Polynomials of degree four. 2. Power polynomials up to a linear change of variables. 3. Chebyshev polynomials up to a linear change of variables.

In each of these cases the fact that the polynomial is invertible by radicals follows from solvability of the corresponding monodromy group or from explicit formulas for its inverse (see for instance~\cite{[2]}).

\subsubsection{Polynomials invertible by $k$-radicals}\label{ssec2.15}

In this section we discuss the following generalization of J.F.Ritt's Theorem.

\begin{thm}\label{thm2.40}(see~\cite{[2]})
 A polynomial invertible by radicals and solutions of equations of degree at most $k$ is a composition of power polynomials, Chebyshev polynomials, polynomials of degree at most $k$ and, if $k\leq 14$, certain primitive polynomials whose inverse have  exceptional monodromy groups. A description of these exceptional polynomials can be given explicitly.
\end{thm}

The proofs rely on classification of monodromy groups of inverse of primitive polynomials  obtained by M\"{u}ller based on group-theoretical results of Feit and on previous work on primitive polynomials whose inverse have exceptional monodromy groups by many authors. Besides the references to these highly involved and technical results an outline of the proof of Theorem 40 is not complicated and it resembles the outline of the proof of Ritt's Theorem.

Let us start with some background on representability by $k$-radicals.

\begin{defn}
Let $k$ be a natural number. A field extension $L/K$ is $k$-radical if there exists a tower of extensions $K=K_0\subset K_1\subset \ldots \subset K_n$ such that $L\subset K_n$ and for each $i$, $K_{i+1}$ is obtained from $K_i$ by adjoining an element $a_i$, which is either a solution of an algebraic equation of degree at most $k$ over $K_i$, or satisfies $a_i^m=b$ for some natural number $m$ and $b\in K_i$.
\end{defn}

\begin{thm}\label{thm2.41}(see~\cite{[6]}) A Galois extension $L/K$ of fields of characteristic zero is $k$-radical if and only if its Galois group is $k$-solvable.
\end{thm}

An algebraic function $z=z(x)$ of one or several complex variables is said to be {\it representable by  $k$-radicals} if the corresponding extension of the field of rational functions is a $k$-radical extension.

Theorem~\ref{thm42} and C. Jordan's Theorem (see sections~\ref{ch2sec1ss1} and ~\ref{ch2sec1ss2}) imply the following corollary.

 \begin{cor}\label{cor2.42}  An algebraic function is representable by $k$-radicals if and only if its monodromy group is $k$-solvable.
 \end{cor}
 (Note that Theorem~\ref{thm2.28} above coincides with a  part of  Theorem~\ref{thm2.41}).

Let us outline briefly the main steps in the proof of Theorem~\ref{thm2.40}:

{\sc Outline of proof of Theorem~\ref{thm2.40}}:

1) Exactly as  in Ritt's theorem one can show that a composition of polynomials is invertible by $k$-radicals if and only if each polynomial in the composition is invertible by $k$-radicals. Thus one can reduce Theorem 40  to the case of primitive polynomials.
\smallskip

2) Feit and Jones totally classified all primitive permutation groups of $n$ elements containing a full cycle.
\smallskip

3) Using  this classification  and Riemann--Hurwitz formula, M\"{u}ller listed
all groups of permutations of $n$ elements  which are monodromy groups of inverses of degree $n$ primitive polynomials.
\smallskip

4) For each group from M\"{u}ller's list of groups of permutations of $n$ elements one can determine the smallest $k$ for which it is
$k$-solvable and choose the {\it exceptional groups} for which $k$ is smaller than $n$.
\smallskip

5) For each such exceptional group one can explicitly describe polynomials whose inverse has the exceptional monodromy group.

\subsection{Multidimensional Topological Galois Theory}\label{ch2sec3}
\subsubsection{Introduction}
In this section we present an outline of the multidimensional  version of topological Galois theory.
The presentation is based on the  book~\cite{[6]}. It contains definitions, statements of results and comments on them. Basically no proofs are presented.

In  topological Galois theory for functions of one variable (see section 2 and ~\cite{[6]}),  it is proved that
the way the Riemann surface of a function is positioned over the complex line can obstruct
the representability of this function ``in finite terms'' (i.e. its representability by radicals, by quadratures, by generalized quadratures and so on).
This not only explains why many algebraic and differential equations are not solvable in finite terms,
but also gives the strongest known results on their unsolvability.

In the multidimensional version of topological Galois theory  analogous results are proved. But in the multidimensional case all constructions and proofs are much more complicated and involved than in the one dimensional case (see~\cite{[6]}).

\subsubsection{Classes of functions}

An equation is solvable ``in finite terms'' (or is solvable ``explicitly'') if its solutions belong to a certain class of functions. Different classes of functions correspond to different notions of  solvability in finite terms.

A class of functions can be introduced by specifying a list of {\it basic functions}
and a list of {\it admissible operations}.
Given the two lists, the class of functions is defined as the set
of all functions that can be obtained
from the basic functions by repeated application of admissible operations.
Below, we define Liouvillian classes of functions  in exactly this way.

Classes of functions, which appear in the problems of integrability in finite terms,
contain multivalued functions.
Thus the basic terminology should be made clear.

We understand operations on multivalued functions
of several variables  in a slightly more
restrictive sense than operations on multivalued functions of
single variable (the one dimensional case is discussed  in section~\ref{ch2sec2} and in~\cite{[6]}.

 Fix a class of basic functions and some set of admissible operations. Can a given
function (which is obtained, say, by solving a certain
algebraic or a differential equation) be expressed through the
basic functions by means of admissible operations? We are
interested in various {\it single valued branches} of
multivalued functions over various domains. Every function,
even if it is multivalued, will be considered as a collection
of all its single valued branches. We will only apply
admissible operations (such as arithmetic operations and
composition) to single valued branches of the function over
various domains. Since we deal with analytic functions, it
suffices only to consider small neighborhoods of points as
domains.

We can now rephrase the question in the following way:
{\it can a given function germ at a given point be expressed through the germs of basic functions
with the help of admissible operations?}
Of course, the answer depends on the choice of a point and on the choice of a single valued
germ at this point belonging to the given multivalued function.
It turns out, however, that for the classes of functions interesting to us the desired expression
is either impossible for every germ of a given multivalued function at every point or
the ``same'' expression serves all germs of a given multivalued function at almost every point
of the space.

For  functions of one variable, we use a different, extended definition of operations on
multivalued functions, in which the multivalued function is viewed as a single object. This definition is essentially equivalent to including the operation of analytic continuation in the list of admissible operations on analytic germs (all details can be found in~\cite{[6]}).
For functions of many variables, we need to adopt the more restrictive understanding
of operations on multivalued functions, which is, however, no less (and perhaps even more) natural.

\subsubsection{Specifics of the multidimensional case}\label{ssec2.3.2}

I was always under impression that a full-fledged
multidimensional version of topological Galois theory was
impossible. The reason was that, to construct such a version
for the case of many variables, one would need to have
information on extendability of function germs not only outside
their ramification sets but also along these sets. It seemed
that there was nothing to extract such information from.

To illustrate the problem  consider the following  situation. Let $f$ be a multivalued analytic function on $\Bbb{C}^n$, whose set of singular points is an analytic set $\Sigma_f\subset \Bbb C^n$. Let $ f_a$ be an analytic  germ of  $f$ at a point $a\in \Bbb C^n$. Let $g:(\Bbb C^k,b)\to
(\Bbb{C}^n,a)$ be an analytic map. Consider a germ $\varphi_b$ at the point $b\in \Bbb C^k$ of the composition $ f_a \circ g_b $. One can ask the following questions:

1) Is it true that $\varphi_b$  is a germ of a
multivalued function $\varphi$ on $\Bbb C^k$, whose set of singular points $\Sigma_\varphi$ is contained in a proper analytic subset of $\Bbb C^k$?

2)  Is it true that  the monodromy group $M_\varphi$ of $\varphi$  corresponding to motions  around the  set $\Sigma _\varphi \subset \Bbb C^k$ can be estimated in terns of the monodromy group $M_f$ of $f$  corresponding to motions  around the  set $\Sigma_f\subset \Bbb C^n $? For example, if $M_f$ is a solvable group is it true that $M_\varphi$ also is a solvable group?

If the image $g(\Bbb C^k)$ is not contained in the singular set $\Sigma_f$  then the answers to the both questions are positive: the set $\Sigma _\varphi $ belongs to the analytic set $g^{-1}(\Sigma_f)$ and the group $M_\varphi$ is a subgroup of a certain factor group of $M_f$. These statements are not complicated and can be proved by the same arguments as in the one dimensional topological Galois theory.

Assume that the multivalued function $f$ has an analytic germ  $f_a$ at a point $a$ belonging to the singular set $\Sigma_f$ (some of the germs of the multivalued function $f$ may
appear to be nonsingular at singular points of this function). Assume now that the image $g(\Bbb C^k)$ is  contained in the singular set $\Sigma_f$ and $a=g(b)$. It turns out that for the germ $\varphi_b=f_a\circ g_b$ the answers to the both questions also are positive.
In  this situation all the proofs are  more involved. They use new arguments from multidimensional complex analysis  and from  group theory.

It turns out that  function germs can
sometimes be automatically extended along their ramification
sets (see~\cite{[6]}). That new statement from complex analysis suggests the positive answer to the first question.

To describe the connection between the monodromy group of the
function $f$ and the monodromy groups of the composition $\varphi = f\circ g$, we
introduce and develop the notion of pullback closure for groups (see~\cite{[6]}). The use of this operation, in turn, forces us
to reconsider all arguments we used in the one dimensional version of topological Galois theory. As a result  we obtain a positive answer to the second question.

\bigskip
\subsubsection{Liouvillian classes of multivariate functions}
In this section we define Liouvillian classes of functions  for the case of several variables.
These classes  are defined in the same way as the corresponding classes for functions of one variable (see section~\ref{ch2sec1ss2} and~\cite{[6]}).
The only difference is in the details.

We fix an ascending chain of standard coordinate subspaces of
strictly increasing dimension:
$0\subset\Bbb C^1\subset\dots\subset\Bbb C^n\subset\dots$
with coordinate functions $x_1$, $\dots$, $x_n$, $\dots$ (for
every $k>0$, the functions $x_1$, $\dots$, $x_k$ are coordinate
functions on $\Bbb C^k$).
Below, we define Liouvillian
classes of functions for each of the standard coordinate
subspaces $\Bbb C^k$.

To define Liouvillian classes, we will need the list of basic elementary functions and the list of classical operations.
\medskip

{\bf List of basic elementary functions.}
\begin{enumerate}
\item All complex constants and all coordinate functions
    $x_1$, $\dots$, $x_n$ for every standard coordinate
    subspace $\Bbb C^n$.

\item The exponential, the logarithm and the power
    $x^\alpha$, where $\alpha$ is any complex constant.

\item Trigonometric functions:  sine, cosine, tangent,
    cotangent.

\item Inverse trigonometric functions: arcsine, arccosine,
    arctangent, arccotangent.

 \end{enumerate}

Let us now turn to the list of classical operations on
functions.

\medskip

{\bf List of classical operations.}
\begin{enumerate}
\item {\it Operation of composition} that takes a function $f$
    of $k$ variables and functions $g_1$, $\dots$, $g_k$ of $n$
    variables to the function $f(g_1,\dots,g_k)$ of $n$
    variables.

\item  {\it Arithmetic operations} that take functions $f$
    and $g$ to the functions $f+g$, $f-g$, $fg$ and $f/g$.

\item {\it Operations of partial differentiation with
    respect to independent variables}. For functions of $n$
    variables, there are $n$ such operations: the $i$-th
    operation assigns the function $\frac{\partial
    f}{\partial x_i}$ to a function $f$ of the variables
    $x_1$, $\dots$, $x_n$.

\item  {\it Operation of integration} that takes $k$
    functions $f_1$, $\dots$, $f_k$ of the variables $x_1$,
    $\dots$, $x_n$, for which the differential one-form
    $\alpha=f_1 dx_1+\dots+f_k dx_k$ is closed, to the
    indefinite integral $y$ of the form $\alpha$ (i.e. to
    any function $y$ such that $dy=\alpha$). The function
    $y$ is determined by the functions $f_1$, $\dots$,
    $f_k$ up to an additive constant.

\item {\it Operation of solving an algebraic equation} that
    takes functions $f_1,\dots,f_n$ to the function $y$
    such that $y^n+f_1y^{n-1}+\dots+f_n=0$. The function
    $y$ may not be quite uniquely determined by the
    functions $f_1$, $\dots$, $f_n$, since an algebraic
    equation of degree $n$ can have $n$ solutions.

\end{enumerate}

We now resume defining Liouvillian classes of functions.
\medskip
\paragraph{Functions of $n$ variables representable by radicals.}
List of basic functions: All complex constants and all coordinate functions.
List of admissible operations:  composition,
arithmetic operations and the operation of taking the $m$-th root
$f^{\frac {1}{m}}$, $m=2,3,\dots$, of a given function~$f$.
\medskip

\paragraph{Functions of $n$ variables representable by
$k$-radicals.} This class of functions is defined in the same
way as the class of functions representable by radicals. We
only need to add the operation of solving algebraic equations
of degree $\leq k$ to the list of admissible operations.

\medskip

\paragraph{ Elementary functions of $n$ variables.}
List of basic functions: basic elementary functions.
List of admissible operations: composition, arithmetic
operations, differentiation.
\medskip

\paragraph{Generalized elementary functions of $n$
variables.} This class of functions is defined in the same way
as the class of elementary functions. We only need to add the
operation of solving algebraic equations to the list of
admissible operations.

\medskip

\paragraph{ Functions of $n$ variables representable by
quadratures.}
List of basic functions: basic elementary functions.
List of admissible operations: composition, arithmetic
operations, differentiation, integration.
\medskip

\paragraph{ Functions of $n$ variables representable by
$k$-quadratures.} This class of functions is defined in the
same way as the class of functions representable by
quadratures. We only need to add the operation of solving
algebraic equations of degree at most $k$ to the list of
admissible operations.

\medskip
\paragraph{ Functions of $n$ variables representable by
generalized quadratures.} This class of functions is defined in
the same way as the class of functions representable by
quadratures. We only need to add the operation of solving
algebraic equations to the list of admissible operations.

\subsubsection{Strong non representability in finite terms}\label{ssec3.3.5}
Topological obstructions to  the representability of functions in finite terms relate to
branching. It turns out that if a function does not belong to a certain Liouvillian class by topological reasons then it automatically does not belong to a much wider {\it extended Liouvillian class of functions}.

Such an extended Liouvillian class is defined as follows: its list of admissible operations is the same as in the original Liouvillian class and its list of basic functions is the  list of basic function in the original class extended by all single valued functions of any number of variables having a proper analytic set of singular points.

\begin{defn}A germ $f$ is a germ of function belonging to the {\it extended class of functions representable by  by quadratures}  if it can be represented  by germs of basic elementary functions and by germs of single valued functions, whose set of singular points is a proper analytic set,  by means of composition, integration,  arithmetic  operations and
differentiation.
\end{defn}

\begin{defn} A germ $f$ is {\it strongly non representable by quadratures}  if it is not a germ of function from the extended class of functions representable by  by quadratures.
\end{defn}

The definition of {\it strong non representability of a germ $f$} by radicals, by $k$-radical, by elementary functions, by generalized elementary functions, by $k$-quadratures and by generalized quadratures is similar to the above definition.

\subsubsection{Holonomic systems of linear differential equations}\label{ssec3.3.6}
Consider a system of $N$ linear
differential equations $L_j(y)=0$, $j=1,\dots, N$,
\begin{eqnarray}
L_j(y)=\sum a_{i_1,\dots, i_n}\frac{\partial^{i_1+\dots
+i_n}y}{\partial x_1^{i_1} \dots
\partial x^{i_n}_n} =0, \label{eq3.7}
\end{eqnarray}
of an unknown function $y$, whose coefficients $a_{i_1,\dots, i_n}$
are analytic  functions in a domain $U\subset \Bbb C^n$.

The system (\ref{eq3.7}) is {\it holonomic} if at every point $a\in U$ the $\Bbb C$-linear space  $V_a$ of germs $y_a$ satisfying the system (\ref{eq3.7})  has finite dimension, $\dim_{\Bbb C}V_a=d(a)<\infty$.
Holonomic systems  can be considered as a multidimensional generalization of linear differential equation  on one unknown function of a single variable. Kolchin obtained a generalization of the Picard--Vessiot theory (Galois theory for linear differential equations)
to the case of holonomic systems of differential equations~\cite{[10]}.

The holonomic system (\ref{eq3.7})  has the following properties:

1) There exists an analytic {\it singular hypersurface  $\Sigma \subset U$} such that the dimension $d(a)=\dim_{\Bbb C}V_a$ is constant $d(a) \equiv d$ on $U\setminus \Sigma$.

2) Let $\gamma:I\to U\setminus \Sigma$ be a continuous map,  where $I$ is
 the unit segment $0\leq t\leq 1$ and $\gamma(0)=a$, $\gamma(1)=b$. Then the space $V_a$ of solutions of (\ref{eq3.7})  at the point $a$  admits  analytic continuation  along $\gamma$ and the space obtained by the continuation at the point $b$ is the space $V_b$ of solutions of (\ref{eq3.7})   at the point $b$.

3) If all  equations of the system (\ref{eq3.7})  admit analytic continuation  to some domain $W$, then the system  obtained by such a continuation  is a holonomic system  in the domain $W$.

Let $a\notin \Sigma$ be a point not belonging to the hypersurface $\Sigma$.
Take an arbitrary path $\gamma(t)$ in the domain
$U$ originating and terminating at $a$ and avoiding the
hypersurface $\Sigma$.
Solutions of this system admit analytic continuations along the path $\gamma$,
which are also solutions of the system.
Therefore, every such path $\gamma$ gives rise to a linear map $M_{\gamma}$
of the solution space $V_a$ to itself.
The collection of linear transformations $M_{\gamma}$ corresponding to all paths
$\gamma$ form a group, which is called the
{\it monodromy group of the holonomic system}.

\subsubsection{$\mathcal {SC}$-germs}\label{ssec3.3.7}
There is a wide class of $\mathcal S$-functions in one variable  containing all Liouvillian functions and stable under classical operations, for which the
monodromy group is defined. The class of $\mathcal S$-functions plays an important role in the one dimensional  version of topological Galois theory (see~\cite[sec.~2]{[6]}). Is there a sufficiently wide class of multivariate function germs   with  similar
properties?

For a long time, I thought that the answer to this
question was negative. In this section the class of
$\mathcal{SC}$-germs is defined. This provides an affirmative answer to
this question. \smallskip

A subset $A$ in a connected $k$-dimensional analytic manifold
$Y$ is called {\it meager} if there exists a countable set of
open domains $U_i\subset M$ and a countable collection of
proper analytic subsets $A_i\subset U_i$ in these domains such
that $A\subset\bigcup A_i$.

The following definition plays a key role in what follows.

\begin{defn} A germ $f_a$ of an analytic function
at a point $a\in \Bbb C^n$ is an {\it
$\mathcal {SC}$-germ} if the following condition is fulfilled.
For every connected complex analytic manifold $Y$, every
analytic map $G: Y\to \Bbb C^n$ and every preimage $b$ of
the point $a$, $G(b)=a$, there exists a meager set $A\subset Y$
such that, for every path $\gamma: [0,1]\to Y$ originating at
the point $b$, $\gamma(0)=b$ and intersecting the set $A$ at
most at the initial moment, $\gamma (t)\not\in A$ for $t>0$,
the germ $f_a$ admits an analytic continuation along the path
$G\circ\gamma: [0,1]\to \Bbb C^n$.
\end{defn}
The following lemma is obvious.
\begin{lem}\label{lem3.43}
The class of $\mathcal SC$-germs contains all germs of analytic functions on $\Bbb C^N\setminus \Sigma$ where $\Sigma$ is an analytic subset in $\Bbb C^N$ where $N$ is a natural number. In particular the class contains all analytic germs of
$\mathcal S$-functions of one variable and all germs of
meromorphic functions of many variables.
\end{lem}

The proof of the following Theorem~\ref{thm3.44} uses the results on extendability of
multivalued analytic functions along their singular point sets
(see~\cite{[6]}).

\begin{thm}\label{thm3.44} (on stability of the class of $\mathcal {SC}$-germs)
The class of $\mathcal {SC}$-germs on $\Bbb C^n$ is stable under the
operation of taking the composition with $\mathcal {SC}$-germs of
$m$-variable functions,
the operation of differentiation and integration. It is stable under
solving algebraic equations whose coefficients are $\mathcal {SC}$-germs and under solving holonomic systems of linear  differential equations whose coefficients are $\mathcal {SC}$-germs.
\end{thm}

Theorem~\ref{thm3.44} implies the following corollary.
\begin{cor}\label{cor3.45} If a germ $f$ is not an $\mathcal {SC}$-germ  then $f$ is strongly non representable by generalized quadratures.  In particular it cannot be a germ of a function belonging to a certain Liouvillian class.
\end{cor}

\subsubsection{Monodromy group of a $\mathcal {SC}$-germs}\label{ssec3.3.8}

The {\it monodromy group and the  monodromy pair} of a $\mathcal {SC}$-germ $f_a$ can by defined in same way as for $\mathcal S$-functions of one variable. By definition the set $\Sigma\subset \Bbb C^n$ of singular points of  $f_a$ is a  meager set. Take any point $x_0\in \Bbb C^n\setminus \Sigma$ and consider the action of the fundamental group $\pi_1 (\Bbb C^n\setminus \Sigma,x_0)$ on the set $F_{x_0}$ of all germs equivalent to  the germ $f_a$. The {\it monodromy group} of $f_a$ is the image of the fundamental group under this action. The {\it monodromy pair} of $f_a$ is the pair $[\Gamma,\Gamma_0]$ where $\Gamma$ is the monodromy group and $\Gamma_0$ is the stationary subgroup of a germ  $f\in F_{x_0}$. Up to an isomorphism the monodromy group and the monodromy pair are independent of a choice  of the point $x_0$ and the germ $f$.

\begin{remark} If a  $\mathcal {SC}$-germ $f_a$ is defined at a singular point  $a\in \Sigma$  then  the {\it monodromy group of $f_a$ along  $\Sigma$ } is  defined: one can  consider continuations of $f_a$  along curves $\gamma$ belonging to $\Sigma$ and define a singular set $\Sigma_1\subset \Sigma$ for $f_a$ along $\Sigma$.  The monodromy group  of $f_a$ along $\Sigma$  corresponds to the action the fundamental group of $\pi_1(\Sigma\setminus \Sigma_1,x_1)$  on the set of germs at $x_1\in \Sigma \setminus \Sigma_1$ obtained by continuation of $f_a$ along $\Sigma$. If the point $a$ belongs to $\Sigma_1$ then one can define also a monodromy group of $f_a$ along $\Sigma_1$ and so on. Thus in the multidimensional case one can associate to an $\mathcal {SC}$-germ an hierarchy of monodromy groups. All these monodromy groups (and corresponding monodromy pairs)  appear in multidimensional topological Galois theory. But the monodromy group and the monodromy pair we discuss above are  most important for our purposes.
\end{remark}

\subsubsection{Stability of certain classes of $\mathcal{SC}$-germs}\label{ssec3.9}
One can prove the following  theorems.
\begin{thm}(see~\cite{[6]})\label{thm3.46}   The class of all
$\mathcal {SC}$-germs, having a solvable monodromy  group is stable
under composition, arithmetic operations, integration and differentiation.This class contains all germs of basic elementary functions and all  germs of single valued  functions whose set of singular points is a proper analytic set.
\end{thm}

\begin{thm}(see~\cite{[6]})\label{thm3.47}  The class of all
$\mathcal {SC}$-germs, having a $k$-solvable monodromy pair  (see~\cite[sec.~2]{[6]}) is stable
under composition, arithmetic operations, integration, differentiation and solution of algebraic equations of degree at most $k$. This class contains all germs of basic elementary functions and all  germs of single valued  functions whose set of singular points is a proper analytic set.
\end{thm}

\begin{thm} (see~\cite{[6]})\label{thm3.48}  The class of all
$\mathcal {SC}$-germs, having an almost solvable monodromy pair  (see~\cite[sec.~2]{[6]})  is stable
under composition, arithmetic operations, integration, differentiation and solution of algebraic equations. This class contains all germs of basic elementary functions and all  germs of single valued  functions whose set of singular points is a proper analytic set.
\end{thm}

Theorems~\ref{thm3.46} --  \ref{thm3.46} imply  the  following corollaries.\\

{\sc Result on  quadratures} {\it If the monodromy group of a $\mathcal {SC}$-germ $f$ is not solvable, then $f$ is   strongly non representable by quadratures.}\\

{\sc Result on  $k$-quadratures} {\it If the monodromy pair of a $\mathcal {SC}$-germ $f$ is not $k$-solvable, then $f$ is   strongly non representable by $k$-quadratures.}\\

{\sc Result on  generalized quadratures} {\it If the monodromy pair  of a $\mathcal {SC}$-germ $f$ is not  almost solvable, then $f$ is   strongly non representable by generalized quadratures.}\\

\subsubsection{Solvability and non solvability of algebraic equation}\label{ssec3.3.10}
Consider an irreducible  algebraic equation
\begin{eqnarray}
 P_n y^n+P_{n-1}y^{n-1}+\dots +P_0=0 \label{eq3.8}
 \end{eqnarray}
  whose coefficients  $P_n,\dots, P_0$ are polynomials  of $N$ complex variables $x_1,\dots, x_N$. Let $\Sigma\subset \Bbb C^N$ be the singular set of the equation (\ref{eq3.8})
defined by the equation $P_nJ=0$ where $J$ is the discriminant of the polynomial (\ref{eq3.8}).

\begin{thm}\label{thm3.49} (see~\cite[sec.~2]{[6]}, \cite{[4]}) Let $y_{x_0}$ be a germ of analytic function at a point $x_0\in \Bbb C^N\setminus \Sigma$ satisfying the equation {\rm  (\ref{eq3.8})}. If the monogromy group of the equation {\rm  (\ref{eq3.8})} is solvable (is $k$-solvable) then the germ $y_{x_0})$ is representable by radicals (is representable by $k$-radicals).
\end{thm}

According to Camille Jordan's theorem (see \cite{[4]})  the  Galois group of the equation (\ref{eq3.8}) over the field $\mathcal R$ of rational functions of $x_1,\dots,x_N$ it is isomorphic to the  monodromy group of  this equation (\ref{eq3.8}). Thus Theorem~\ref{thm3.49}follows from Galois theory (see~\cite{[6]}, \cite{[4]}).

\begin{thm}\label{thm3.50}(see~\cite{[6]}) Let $y_{x_0}$ be a germ of analytic function at a point $x_0\in \Bbb C^N$ satisfying the  equation {\rm  (\ref{eq3.8})}. If the monodromy group of the equation is not solvable (is not $k$-solvable) then the germ $y_{x_0}$ is strongly non representable by quadratures (is strongly non representable by $k$-quadratures).
\end{thm}

Theorem~\ref{thm3.50} follows from the results on quadratures and on $k$-quadratures from the previous section.

Consider the universal degree $n$ algebraic function $y(a_n,\dots, a_0)$ defined by the equation
\begin{eqnarray}
a_ny^n+\dots+a_0=0. \label{eq3.9}
\end{eqnarray}

It is easy to see that the monodromy group of the equation (\ref{eq3.9}) is isomorphic to the group $S_n$ of all permutations of $n$ element. For $n\geq 5$ the group $S_n$ is unsolvable and it is not $k$-solvable group for $k<n$. Thus Theorem~\ref{thm3.50} implies the following strongest known version of the Abel-–Ruffini Theorem.

\begin{thm}\label{thm3.51}({\sc a version of the Abel--Ruffini Theorem}) Let $y_{a}$ be a germ of analytic function at a point $a$ satisfying the universal degree $n\geq 5$ algebraic equation. If $n\geq 5$ then the germ $y_{a}$ is strongly non representable by $(n-1)$ quadratures. In particular the germ $y_{a}$ is strongly non representable by quadratures.
\end{thm}

\subsubsection{Solvability and non solvability of  holonomic systems of linear differential equations}\label{ssec3.3.11}  Consider a system of $N$ linear
differential equations $L_j(y)=0$, $j=1,\dots, N$,
\begin{eqnarray}
L_j(y)=\sum a_{i_1,\dots, i_n}\frac{\partial^{i_1+\dots
+i_n}y}{\partial x_1^{i_1} \dots
\partial x^{i_n}_n} =0, \label{eq3.10}
\end{eqnarray}
on an unknown function $y$, whose coefficients $a_{i_1,\dots, i_n}$
are rational functions of $n$ complex variables $x_1$, $\dots$, $x_n$.
Assume that the system (\ref{eq3.10}) is holonomic in $\Bbb C^n\setminus \Sigma_1$ where $\Sigma_1$ is the union of poles of the coefficients $a_{i_1,\dots, i_n}$. Let $\Sigma_2\subset \Bbb C^n\setminus \Sigma_1$ be the  singular hypersurface of a holonomic system (\ref{eq3.10}).

Every germ $y_a$ of a solution of the system at a point $a\in \Bbb C^n \backslash\Sigma$ where $\Sigma=\Sigma_1\cup \Sigma_2$ admits an analytic continuation along every path
avoiding the hypersurface $\Sigma$ so the monodromy group of the system ((\ref{eq3.10}) is well-defined.

\begin{thm}\label{thm3.52} (see~\cite{[6]})
If the monodromy group of the holonomic system {\rm (\ref{eq3.10})} is not solvable (not $k$-solvable, not almost solvable), then a germ $y_a$ of almost every solution at a point $a\in \Bbb C^n\setminus \Sigma$ is strongly non representable by quadratures (is strongly non representable by $k$-quadratures, is strongly non representable by generalized quadratures).
\end{thm}

Theorem~\ref{thm3.52} follows from the results on quadratures, on $k$-quadratures and on generalized quadratures from  section~\ref{ssec3.9}.

A holonomic system is said to be
{\it regular}, if near the singular set $\Sigma$
and near infinity the solutions of the system grow at most polynomially.

\medskip

\begin{thm}\label{thm3.53}  (see~\cite{[6]}) If the monodromy group of a regular holonomic system  is  solvable (is $k$-solvable, is almost solvable), then a germ $y_a$ of almost every solution at a point $a\in \Bbb C^n\setminus \Sigma$ is  representable by quadratures (is  representable by $k$-quadratures, is  representable by generalized quadratures).
\end{thm}

\subsubsection{ Completely integrable systems of linear differential equations with small coefficients}\label{ssec3.12} Consider a completely integrable system of linear differential equations of the
following form
\begin{eqnarray}
 dy=Ay,\label{eq3.11}
\end{eqnarray}
where $y=y_1$, $\dots$, $y_N$ is an unknown vector-function, and $A$ is a
$(N\times N)$-matrix consisting of differential one-forms with rational coefficients
on the space $\Bbb C^n$ satisfying the condition of complete integrability
$dA+A\wedge A=0$ and having the following form:
 $$
 A=\sum _{i=1}^kA_i\frac{dl_i}{l_i},
 $$
where $A_i$ are constant matrices, and $l_i$ are linear (not necessarily homogeneous)
functions on $\Bbb C^n$.

If the matrices $A_i$ can be simultaneously reduced to the
triangular form, then system (\ref{eq3.11}), as any completely integrable
triangular system, is solvable by quadratures. Of course, there
exist solvable systems that are not triangular. However, if the
matrices $A_i$ are sufficiently small, then there are no such
systems. Namely, the following theorem holds.

\begin{thm}\label{thm3.54} (see~\cite{[6]})
A system  (\ref{eq3.11}) that does not reduce to the triangular form and
such that the matrices $A_i$ have sufficiently small norms is
unsolvable by generalized quadratures in the following strong sense. At every point $a\in \Bbb C^n$ where the matrix $A$ is regular, and for almost any germ $y_a=(y_1,\dots,y_N)_a$ of a vector-function satisfying the system (\ref{eq3.11}), there is a component $(y_i)_a$ which is strongly non representable by generalized quadratures.
\end{thm}

Multidimensional Theorem~\ref{thm3.54} is similar to the one dimensional Corollary 38. Their proofs (see~\cite{[6]}) are also similar.  We only need to replace the reference to the (one-dimensional)
Lappo-Danilevsky theory with the reference to the
multidimensional version of it from~\cite{[9]}.

\addcontentsline{toc}{section}{References}
\bibliographystyle{alpha}
\bibliography{khov}

\newcommand{\SortNoop}[1]{}
\begin{thebibliography}{{\SortNoop{Put}}dPS03}

\bibitem[Ale04]{[1]}
V.~B. Alekseev.
\newblock {\em Abel's theorem in problems and solutions}.
\newblock Kluwer Academic Publishers, Dordrecht, 2004.
\newblock Based on the lectures of Professor V. I. Arnold, With a preface and
  an appendix by Arnold and an appendix by A. Khovanskii.

\bibitem[BK16]{[2]}
Y.~Burda and A.~Khovanskii.
\newblock Polynomials invertible in {$k$}-radicals.
\newblock {\em Arnold Math. J.}, 2(1):121--138, 2016.

\bibitem[Kho70]{[3]}
A.~G. Khovanski\u{\i}.
\newblock The representability of algebroidal functions by superpositions of
  analytic functions and of algebroidal functions of one variable.
\newblock {\em Funkcional. Anal. i Prilo\v{z}en.}, 4(2):74--79, 1970.

\bibitem[Kho71]{[4]}
A.~G. Khovanski\u{\i}.
\newblock Superpositions of holomorphic functions with radicals.
\newblock {\em Uspehi Mat. Nauk}, 26(3(159)):213--214, 1971.

\bibitem[Kho91]{[5]}
A.~G. Khovanski\u{\i}.
\newblock {\em Fewnomials}, volume~88 of {\em Translations of Mathematical
  Monographs}.
\newblock American Mathematical Society, Providence, RI, 1991.
\newblock Translated from the Russian by Smilka Zdravkovska.

\bibitem[Kho14]{[6]}
Askold Khovanskii.
\newblock {\em Topological {G}alois theory}.
\newblock Springer Monographs in Mathematics. Springer, Heidelberg, 2014.
\newblock Solvability and unsolvability of equations in finite terms,
  Appendices C and D by Khovanskii and Yuri Burda [Yura Burda on title page
  verso], Translated from the Russian by V. Timorin and V. Kirichenko (Chapters
  1--7) and Lucy Kadets (Appendices A and B).

\bibitem[Kho18a]{[8]}
Askold Khovanskii.
\newblock Integrability in finite terms and actions of {Lie} groups.
\newblock {\em Uspehi Mat. Nauk}, 2018.
\newblock To appear.

\bibitem[Kho18b]{[7]}
Askold Khovanskii.
\newblock Solvability of equations by quadratures and {N}ewton's theorem.
\newblock {\em Arnold Math. J.}, 4(2):193--211, 2018.

\bibitem[Lek91]{[9]}
V.~P. Leksin.
\newblock The {R}iemann-{H}ilbert problem for analytic families of
  representations.
\newblock {\em Mat. Zametki}, 50(2):89--97, 1991.

\bibitem[{\SortNoop{Put}}dPS03]{[10]}
Marius {\SortNoop{Put}}van~der Put and Michael~F. Singer.
\newblock {\em Galois theory of linear differential equations}, volume 328 of
  {\em Grundlehren der Mathematischen Wissenschaften [Fundamental Principles of
  Mathematical Sciences]}.
\newblock Springer-Verlag, Berlin, 2003.

\bibitem[Rit22]{[12]}
J.~F. Ritt.
\newblock On algebraic functions which can be expressed in terms of radicals.
\newblock {\em Trans. Amer. Math. Soc.}, 24(1):21--30, 1922.

\bibitem[Rit48]{[11]}
Joseph~Fels Ritt.
\newblock {\em Integration in {F}inite {T}erms. {L}iouville's {T}heory of
  {E}lementary {M}ethods}.
\newblock Columbia University Press, New York, N. Y., 1948.

\bibitem[Ros73]{[13]}
Maxwell Rosenlicht.
\newblock An analogue of l'{H}ospital's rule.
\newblock {\em Proc. Amer. Math. Soc.}, 37:369--373, 1973.

\end{thebibliography}
\end{document}